\documentclass[11pt]{amsart}

\usepackage[T1]{fontenc}
\usepackage{lmodern}
\usepackage[margin=1in]{geometry}
\usepackage{amsmath,amssymb,amsthm,mathtools}
\usepackage{aliascnt}
\usepackage{booktabs}
\usepackage{array,tabularx}
\usepackage{enumitem}
\usepackage{needspace}
\usepackage{tikz}
\usetikzlibrary{arrows.meta,calc,positioning}
\usepackage[nocompress]{cite}
\usepackage[hidelinks]{hyperref}
\usepackage[nameinlink,noabbrev]{cleveref}
\usepackage{microtype}

\newtheorem{theorem}{Theorem}[section]
\newaliascnt{lemma}{theorem}
\newtheorem{lemma}[lemma]{Lemma}
\aliascntresetthe{lemma}
\newaliascnt{proposition}{theorem}
\newtheorem{proposition}[proposition]{Proposition}
\aliascntresetthe{proposition}
\newaliascnt{corollary}{theorem}

\aliascntresetthe{corollary}
\newaliascnt{remark}{theorem}
\newtheorem{remark}[remark]{Remark}
\aliascntresetthe{remark}
\newaliascnt{definition}{theorem}
\newtheorem{definition}[definition]{Definition}
\aliascntresetthe{definition}

\crefname{lemma}{lemma}{lemmas}
\Crefname{lemma}{Lemma}{Lemmas}
\crefname{proposition}{proposition}{propositions}
\Crefname{proposition}{Proposition}{Propositions}
\crefname{corollary}{corollary}{corollaries}
\Crefname{corollary}{Corollary}{Corollaries}
\crefname{remark}{remark}{remarks}
\Crefname{remark}{Remark}{Remarks}
\crefname{definition}{definition}{definitions}
\Crefname{definition}{Definition}{Definitions}

\newcommand{\R}{\mathbb R}
\newcommand{\T}{\mathcal T}
\newcommand{\Q}{\mathbb Q}
\newcommand{\diver}{\operatorname{div}}

\newcommand{\supp}{\operatorname{supp}}
\newcommand{\mean}{\mathsf M}

\newcommand{\one}{\boldsymbol 1}

\newcommand{\norm}[1]{\lVert #1\rVert}

\title[Scott--Vogelius Stability on 3D Freudenthal Meshes]
{Uniform Stability of Scott--Vogelius Elements on Three-Dimensional Freudenthal Meshes in Degrees Four and Five: Resolving the Farrell--Mitchell--Scott Conjecture}

\newcommand{\ManuscriptAuthorNames}{Hanbing Liang and Fujun Liu}
\newcommand{\SharedAffiliation}{Nanophotonics and Biophotonics Key Laboratory
of Jilin Province, School of Physics, Changchun University of Science and
Technology, Changchun 130022, China}

\author{Hanbing Liang and Fujun Liu}
\address{\SharedAffiliation}
\email{fjliu@cust.edu.cn}
\thanks{Corresponding author: Fujun Liu.}
\hypersetup{
  pdftitle={Uniform Scott--Vogelius Stability on Three-Dimensional
    Freudenthal Meshes in Degrees Four and Five},
  pdfauthor={\ManuscriptAuthorNames},
  pdfsubject={Uniform stability of the three-dimensional Scott--Vogelius
    finite element on Freudenthal meshes},
  pdfkeywords={Scott--Vogelius finite elements, inf--sup stability,
    Freudenthal triangulations, divergence-free finite elements,
    macroelement methods}
}
\date{}

\begin{document}
\raggedbottom
\begin{abstract}
We establish a uniform inf-sup stability estimate for the Scott-Vogelius finite element spaces on uniform Freudenthal tetrahedralizations of the unit cube for polynomial degrees $k \geq 4$. This result completely settles the first conjecture of Farrell, Mitchell, and Scott for the critical degrees $k=4$ and $k=5$, complementing the known stability range for higher polynomial degrees. The main mathematical difficulties stem from the complex topological compatibility required at the singular vertices and the corresponding mean-value constraints across adjacent elements. We tackle these challenges by developing a unified barycentric skeleton-bubble calculus that explicitly constructs vertex jets, edge modes, and face transfers to globally route element means. The accompanying exact computations independently verify these finite-dimensional identities and provide reproducibility data.

\end{abstract}

\maketitle

\section{Introduction}\label{sec:introduction}

The exact preservation of the divergence-free constraint is of fundamental importance in the stable numerical simulation of incompressible flows. The significance of this exact constraint for incompressible flow discretizations, which eliminates pressure-driven non-physical artifacts, is extensively surveyed in \cite{JohnEtAl2017}. The finite element spaces analyzed in this paper belong to the Scott and Vogelius family introduced in \cite{ScottVogelius1985Continua}. The image space definition inherently incorporates the mesh dependent compatibility conditions for the discrete divergence directly into the pressure space. The corresponding range characterization and bounded right inverse problem were originally developed in \cite{Vogelius1983,ScottVogelius1985RightInverse}. For the modern two dimensional theory in degree four and above, we refer the reader to \cite{GuzmanScott2019}. Related conforming and exactly divergence-free constructions in three dimensions include those presented in \cite{Zhang2005,GuzmanNeilan2014,GuzmanNeilan2018,NeilanSap2016}. Their Stokes complex structure and pointwise mass conservation are discussed in \cite{FalkNeilan2013,Neilan2020}. A dimension based comparison of three dimensional exactly divergence-free spaces is given in \cite{ScottTscherpel2024}. 

To establish our theoretical framework, we put $\Omega=(0,1)^3$. For a given integer $N$, we divide $\Omega$ into $N^3$ closed cubes of side $h=N^{-1}$. If a cube is denoted by $C=x_C+h[0,1]^3$, its six Freudenthal tetrahedra are \[
x_C+h\,\operatorname{conv}
\{0,e_{\sigma(1)},e_{\sigma(1)}+e_{\sigma(2)},(1,1,1)\},
\qquad \sigma\in S_3
\]. The same orientation is used in every cube. The resulting conforming tetrahedral mesh is denoted by $\T_h$. This is the standard Freudenthal subdivision detailed in \cite{Freudenthal1942,Bey2000}. It is shape regular, and every element is a translation, scaling, and coordinate permutation of one of finitely many reference tetrahedra. Consequently, the affine scaling, trace, and inverse estimates used below have constants independent of $h$, as established in \cite{BrennerScott2008}.

For an integer $k\geq1$, we define the velocity and pressure spaces as
\begin{equation}\label{eq:spaces}
 V_{h,k}=\{v\in H^1_0(\Omega;\R^3):v|_T\in P_k(T)^3 \text{ for every }T\in\T_h\}, \qquad Q_{h,k}=\diver V_{h,k}.
\end{equation}
The pressure space is generally a proper subspace of the broken polynomial space \(P_{k-1}(\T_h)\). All constructions below are carried out on this divergence image. We denote the standard gradient norm by $|v|_1=\norm{\nabla v}_{L^2(\Omega)}$ and the pressure norm by $\norm q=\norm q_{L^2(\Omega)}$. Because the velocity field has a zero trace, the gradient seminorm is a valid norm on $V_{h,k}$. Furthermore, the divergence is bounded by this norm, satisfying
\begin{equation}\label{eq:div-bound}
 \norm{\diver v}\leq\sqrt3\,|v|_1.
\end{equation}

For $k\geq6$, Zhang's Theorem 3.1 in \cite[pp.~691-692, (3.69)]{Zhang2011} proves the required uniform inf-sup estimate for the same domain, homogeneous boundary condition, and uniformly oriented mesh. His definitions in \cite[(2.3)-(2.4)]{Zhang2011} correspond exactly to our $V_{h,k}$ and $Q_{h,k}$. The restriction of the divergence operator to the orthogonal complement of its kernel is a bijection onto $Q_{h,k}$, and the inf-sup estimate bounds its inverse uniformly. We prove the stability for the critical lower degrees $k=4$ and $k=5$ below. These are precisely the cases left open in Conjecture 1 of Farrell, Mitchell, and Scott \cite[Conjecture~1]{FarrellMitchellScott2024}.

\begin{theorem}[Main theorem]\label{thm:main}
For every fixed integer $k\geq4$, there exist a constant $C_k$, independent of $N$, and a linear operator $R_{h,k}:Q_{h,k}\longrightarrow V_{h,k}$ such that
\begin{equation}\label{eq:right-inverse}
 \diver R_{h,k}q=q,\qquad |R_{h,k}q|_1\leq C_k\norm q.
\end{equation}
Consequently the reduced Scott and Vogelius inf-sup constant satisfies
\begin{equation}\label{eq:inf-sup}
 \inf_{0\ne q\in Q_{h,k}}\ \sup_{0\ne v\in V_{h,k}} \frac{(\diver v,q)_{L^2}}{|v|_1\norm q}\geq C_k^{-1}
\end{equation}
uniformly in $h$.
\end{theorem}

The uniformly bounded right inverse formulation in \cref{thm:main} is the operator form of the discrete inf-sup condition. The standard mixed method framework is presented in \cite{BoffiBrezziFortin2013}. The right inverse is constructed in four stages, each preserving the invariants established by the preceding stages. The patchwise organization follows the macroelement principle introduced in \cite{Stenberg1990}. The local maps in the present argument are exhibited explicitly and assembled with uniformly bounded overlap. First, an analytic argument using Bogovskii and Scott-Zhang interpolation lifts arbitrary image pressures to zero means on every tetrahedron. Next, edge-jet compatibility and supported cubic vertex-star lifts realize zero broken vertex values while preserving these zero means. Subsequently, supported edge-star lifts followed by fixed-cluster mean routing eliminate edge traces. Finally, an analytic element-bubble isomorphism eliminates the interior residual. For the target degrees $k=4$ and $k=5$, the Freudenthal specific local ingredients are the supported vertex-star lift, the supported edge-star lift, and the fixed-cluster routing family. The first two are given by explicit barycentric formulas, and their exact finite dimensional realizations govern the subsequent proofs.

\section{Local Operators and Preprocessing}\label{sec:preprocess}

The explicit constructions throughout our proof rely on a fundamental barycentric skeleton-bubble calculus. The cubic vertex jets, the face transfers, and the endpoint and middle edge modes below are obtained by choosing the exponents so that exactly the prescribed skeleton trace survives.

\begin{lemma}[Barycentric skeleton-bubble calculus]\label{lem:skeleton-bubble}
If $S=[x_0,\ldots,x_r]$ is a mesh simplex and $\alpha_i\geq1$, then the function
\begin{equation}\label{eq:skeleton-bubble-def}
 b_{S,\alpha}|_T=\prod_{i=0}^r\lambda_{T,x_i}^{\alpha_i} \quad(T\supset S),\qquad b_{S,\alpha}=0\quad(T\not\supset S)
\end{equation}
is continuous and supported on the closed star of $S$. It has zero exterior trace whenever $S$ is not contained in a boundary face, and its gradient satisfies
\begin{equation}\label{eq:skeleton-bubble-rule}
 \nabla b_{S,\alpha} = \sum_{i=0}^r\alpha_i\lambda_{x_i}^{\alpha_i-1} \prod_{j\ne i}\lambda_{x_j}^{\alpha_j}\nabla\lambda_{x_i}.
\end{equation}
On any subsimplex, a summand in \eqref{eq:skeleton-bubble-rule} vanishes unless every remaining barycentric factor is carried by that subsimplex.
\end{lemma}

\begin{proof}
On a common tetrahedron face, either both neighboring elements contain $S$, in which case the barycentric traces agree, or the face omits a vertex of $S$, in which case both traces are zero. This proves conformity and the support statement. The same observation gives the assertion about exterior trace. Formula \eqref{eq:skeleton-bubble-rule} is the standard product rule, and its restriction to a subsimplex proves the last assertion.
\end{proof}

We first construct the two preprocessing operators based on this calculus. Both the element-mean lift and the vertex lift are analytic. The vertex lift combines edge-jet bubbles with face-bubble routing on the six possible vertex stars.

\subsection{Stable Element-Mean and Vertex Lifts}

\begin{lemma}[Element-mean lift]\label{lem:means}
There is a linear map $L_h^0:L^2_0(\Omega)\to V_{h,3}$ such that
\begin{equation}\label{eq:mean-lift}
 \int_T\diver L_h^0q=\int_Tq\quad(T\in\T_h), \qquad |L_h^0q|_1\leq C_0\norm q,
\end{equation}
where $C_0$ is independent of $h$.
\end{lemma}

\begin{proof}
The Bogovskii right-inverse theorem in the form of \cite[Theorem~III.3.1]{Galdi2011} applies because the cube is a bounded Lipschitz domain. Hence there is a linear operator $\mathcal B:L^2_0(\Omega)\to H^1_0(\Omega;\R^3)$ with
\begin{equation}
 \diver\mathcal Bq=q,\qquad |\mathcal Bq|_1\leq C_B\norm q.
\end{equation}
Let $J_h:H^1_0(\Omega)^3\to [C^0P_1(\T_h)]^3\cap H^1_0(\Omega)^3$ be the Scott-Zhang quasi-interpolant. We use the boundary-preserving construction of \cite[Section~2]{ScottZhang1990}, the local estimate \cite[(4.3)]{ScottZhang1990}, and the corresponding global stability \cite[Corollary~4.1, (4.5)]{ScottZhang1990}, specialized to $p=2$, polynomial degree one, and Sobolev orders zero and one. These specific estimates are given by
\begin{equation}\label{eq:SZ}
 |J_hu|_{H^1(T)}\leq C|u|_{H^1(\omega_T)},\qquad \norm{u-J_hu}_{L^2(T)}\leq Ch|u|_{H^1(\omega_T)}.
\end{equation}
Here $\omega_T$ is a fixed-ring element patch. The hypotheses in that theorem are a shape-regular simplicial mesh and homogeneous trace data. Both conditions hold here with constants independent of $h$.

Set $u=\mathcal Bq$. For each interior triangular face $F$, fix a unit normal $n_F$. Let $b_F$ be the continuous piecewise-cubic face bubble on the two tetrahedra incident to $F$. On either incident tetrahedron it is the product of the three barycentric coordinates belonging to the vertices of $F$, and it is zero elsewhere. The two traces agree on $F$ and $b_F$ vanishes on every other face, making $b_Fn_F\in V_{h,3}$. We then put
\begin{equation}
 \delta_F=\int_F(u-J_hu)\cdot n_F,\qquad c_F=\delta_F\big/\int_Fb_F,\qquad L_h^0q=J_hu+\sum_{F\text{ interior}}c_Fb_Fn_F.
\end{equation}
On boundary faces both $u$ and $J_hu$ have zero trace. On an interior face the correction makes the integrated normal flux equal to that of $u$. The divergence theorem therefore proves the first assertion in \eqref{eq:mean-lift} element by element.

Applying the scaled trace inequality to $u-J_hu$ and subsequently using \eqref{eq:SZ} gives
\begin{equation}
 |\delta_F| \leq |F|^{1/2}\norm{u-J_hu}_{L^2(F)} \leq C\bigl(h^{1/2}\norm{u-J_hu}_{L^2(\omega_F)} +h^{3/2}|u-J_hu|_{H^1(\omega_F)}\bigr) \leq C h^{3/2}|u|_{H^1(\omega_F)}.
\end{equation}
On this uniform family, $\int_Fb_F=c_bh^2$ with $c_b>0$ fixed and $|b_Fn_F|_{H^1(\omega_F)}\leq C h^{1/2}$. Hence we bound the correction by
\begin{equation}
 |c_Fb_Fn_F|_{H^1(\omega_F)}\leq C|u|_{H^1(\omega_F)}.
\end{equation}
The face patches have bounded overlap. Together with the first estimate in \eqref{eq:SZ}, this yields $|L_h^0q|_1\leq C|u|_1\leq C_0\norm q$.
\end{proof}

A related element-mean construction appears in Zhang's Lemma 3.1 \cite[pp.~672-673, (3.1)]{Zhang2011}.

For a mesh vertex $a$, let $\T(a)=\{T\in\T_h:a\in T\}$ and let $\omega_a$ denote the union of tetrahedra in $\T(a)$. Write $\mathcal E(a)$ for the geometric mesh edges issuing from $a$ and $\mathcal E^\circ(a)$ for those not contained in $\partial\Omega$. The superscript $\circ$ implies that the relative interior of the edge lies in $\Omega$, not necessarily that both endpoints are interior.

If $T=[a,b_1,b_2,b_3]$ and $\lambda_{T,b_i}$ is the barycentric coordinate of $b_i$, we define the edge-jet map as
\begin{equation}\label{eq:vertex-jet-map}
 A_a:(\R^3)^{\mathcal E^\circ(a)}\longrightarrow\R^{\T(a)},\qquad (A_as)_T=\sum_{i=1}^3\sum_{j=1}^3 \partial_j\lambda_{T,b_i}\,s_{[a,b_i],j},
\end{equation}
where a term is omitted when $[a,b_i]\subset\partial\Omega$. The formula remains strictly independent of the order of $b_1,b_2,b_3$.

\begin{lemma}[Complete compatibility from edge jets]\label{lem:vertex-jets}
For every $k\geq2$, the compatibility constraints are satisfied such that
\begin{equation}\label{eq:vertex-image}
 \left\{\bigl((\diver w)|_T(a)\bigr)_{T\in\T(a)}:w\in V_{h,k}\right\} =\operatorname{im}A_a.
\end{equation}
Thus \eqref{eq:vertex-jet-map} characterizes all vertex compatibility identities in $Q_{h,k}$.
\end{lemma}

\begin{proof}
Let $w\in V_{h,k}$. Continuity makes the restriction of each component $w_j$ to a geometric edge single-valued. Hence the directional derivative
\begin{equation}
 s_{[a,b],j}=\partial_{b-a}(w_j|_{[a,b]})(a)
\end{equation}
is common to every incident tetrahedron. If $[a,b]\subset\partial\Omega$, the homogeneous trace naturally enforces $s_{[a,b],j}=0$. On the tetrahedron $T=[a,b_1,b_2,b_3]$, the three directions $b_i-a$ form a complete basis and yield
\begin{equation}
 \nabla(w_j|_T)(a)=\sum_{i=1}^3s_{[a,b_i],j} \nabla\lambda_{T,b_i}.
\end{equation}
Taking the trace of the vector gradient proves the forward inclusion in \eqref{eq:vertex-image}.

Conversely, we fix $e=[a,b]\in\mathcal E^\circ(a)$. On every tetrahedron containing $e$, we set $\phi_{a,e}=\lambda_a\lambda_b$ and extend it by zero outside the edge star. Its traces agree, and it vanishes on every exterior face of that star. It also vanishes on $\partial\Omega$ because a physical boundary face meeting the support contains at most one endpoint of $e$, since $e$ is not a boundary edge. Along $e$, parametrized from $a$ to $b$, we have $\phi_{a,e}=t(1-t)$, making its directional derivative at $a$ equal to one. Along every other edge issuing from $a$, this derivative is zero. Therefore the vector field
\begin{equation}
 w=\sum_{e\in\mathcal E^\circ(a)}\sum_{j=1}^3 s_{e,j}\phi_{a,e}e_j\in V_{h,2}\subset V_{h,k}
\end{equation}
has precisely the prescribed edge jets. Applying the first half of the proof computes its vertex divergence vector as $A_as$ and strictly proves the reverse inclusion.
\end{proof}

At a boundary vertex, the omitted edge-jet columns correspond precisely to the directional derivatives annihilated by the homogeneous trace condition. Figure \ref{fig:vertex-states} illustrates the six possible vertex locations in the grid box. The wireframe visually represents the entire box rather than the tetrahedron star and distinguishes the two inequivalent types of box corner and boundary edge. In this figure, $r,s,t$ denote strictly interior grid coordinates.

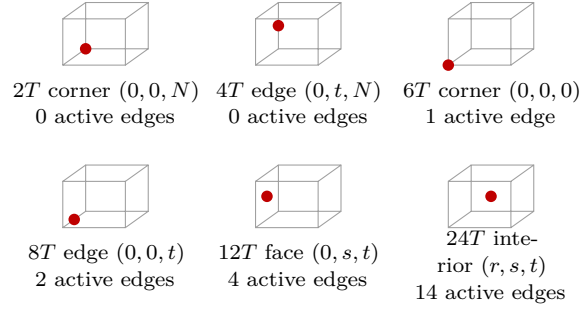
\begin{figure}[t]
\centering
\begin{tikzpicture}[font=\scriptsize]
\newcommand{\vertexpanel}[5]{%
\begin{scope}[shift={(#1,#2)},scale=.72]
\draw[gray!75]

(0,0)--(1.15,0)--(1.15,.85)--(0,.85)--cycle

(.42,.30)--(1.57,.30)--(1.57,1.15)--(.42,1.15)--cycle

(0,0)--(.42,.30) (1.15,0)--(1.57,.30)

(1.15,.85)--(1.57,1.15) (0,.85)--(.42,1.15);
\fill[red!75!black] (#3,#4) circle (3pt);
\node[align=center,text width=2.45cm] at (.78,-.72) {#5};
\end{scope}}
\vertexpanel{0}{2.15}{.42}{.30}{$2T$ corner $(0,0,N)$\\$0$ active edges}
\vertexpanel{2.55}{2.15}{.42}{.725}{$4T$ edge $(0,t,N)$\\$0$ active edges}
\vertexpanel{5.10}{2.15}{0}{0}{$6T$ corner $(0,0,0)$\\$1$ active edge}
\vertexpanel{0}{0}{.21}{.15}{$8T$ edge $(0,0,t)$\\$2$ active edges}
\vertexpanel{2.55}{0}{.21}{.575}{$12T$ face $(0,s,t)$\\$4$ active edges}
\vertexpanel{5.10}{0}{.785}{.575}{$24T$ interior $(r,s,t)$\\$14$ active edges}
\end{tikzpicture}
\caption{Representatives of all vertex states. Coordinate permutations and central inversion generate the other placements. The labels give the size of the incident tetrahedron star and the number of incident edges not lying in the physical boundary.}
\label{fig:vertex-states}
\end{figure}

\begin{lemma}[Six vertex states and finite coverage]\label{lem:vertex-coverage}
Every vertex on every $N^3$ Freudenthal mesh exhibits exactly one of the six states in the table. Up to integer translation, coordinate permutation, and simultaneous reversal of all three coordinates, the local star together with its active-edge set strictly matches the canonical state shown. Consequently the six unit-lattice matrices in \eqref{eq:vertex-jet-map} comprehensively cover arbitrary $N$.
\begin{center}
\small
\begin{tabular}{@{}llrrrr@{}}
\toprule
location & canonical $a$ in $N=3$ & $|\T(a)|$ & $|\mathcal E(a)|$ &
$|\mathcal E^\circ(a)|$ & placements\\
\midrule
two-tet corner &(0,0,3)&2&4&0&6\\
four-tet boundary edge &(0,1,3)&4&6&0&$6(N-1)$\\
six-tet corner &(0,0,0)&6&7&1&2\\
eight-tet boundary edge &(0,0,1)&8&8&2&$6(N-1)$\\
twelve-tet boundary face &(0,1,1)&12&10&4&$6(N-1)^2$\\
interior &(1,1,1)&24&14&14&$(N-1)^3$\\
\bottomrule
\end{tabular}
\end{center}
\end{lemma}

\begin{proof}
We consider one cube containing $a$ and encode the position of $a$ in that cube by $\epsilon\in\{0,1\}^3$. A Freudenthal chain contains that corner exactly when its first $j=|\epsilon|$ coordinate steps are the coordinates on which $\epsilon$ equals one. There are therefore
\begin{equation}\label{eq:corner-chain-count}
 j!(3-j)!
\end{equation}
tetrahedra in that cube containing the corner. At an interior vertex all eight corner words occur, so summing \eqref{eq:corner-chain-count} yields exactly 24. At a boundary-face vertex one bit is fixed and the other two vary, which results in 12 tetrahedra. At a boundary-edge vertex two bits are fixed, where equal fixed bits yield 8 and unequal fixed bits yield 4. At a box corner all bits are fixed, meaning the two constant words give 6 and the other six words give 2. This proves both the six star sizes and the placement counts, which perfectly sum to $(N+1)^3$.

The incident edges are the one, two, or three-step subchains joining $a$ to another corner. Directly listing those subchains computes the two edge-count columns. A listed edge is inactive precisely when its endpoints share a physical coordinate plane. Finally, translations and coordinate permutations preserve the coordinate order subdivision, while central inversion reverses every chain and hence preserves the same set of tetrahedra. These operations carry each row directly to its canonical state. The $N=1$ mesh contains the two corner states, and the $N=2$ mesh contains all six states. Since the supported space below vanishes on the entire boundary of $\omega_a$, the local state involves no additional boundary data.
\end{proof}

The vertex construction exclusively uses the following local statement. Let
\begin{equation}
 S_a=\{v\in[C^0P_3(\T(a))]^3:v=0\text{ on }\partial\omega_a\}.
\end{equation}
Every member extends by zero to $V_{h,3}$. We then define
\begin{equation}
 M_av=\left(\int_T\diver v\right)_{T\in\T(a)},\quad B_av=\bigl((\diver v)|_T(b)\bigr)_{T\in\T(a),\ b\in\mathcal V(T)\setminus\{a\}},\quad D_av=\bigl((\diver v)|_T(a)\bigr)_{T\in\T(a)}.
\end{equation}

\begin{lemma}[Supported vertex-star lift]\label{lem:vertex-local}
For every mesh vertex $a$, we observe that
\begin{equation}\label{eq:vertex-local-image}
 D_a(\ker M_a\cap\ker B_a)=\operatorname{im}A_a.
\end{equation}
There is a fixed linear right inverse $G_a$ on this image such that
\begin{equation}\label{eq:vertex-local-properties}
 D_aG_ad=d,\qquad B_aG_ad=0,\qquad M_aG_ad=0.
\end{equation}
Furthermore, the magnitude is bounded by
\begin{equation}\label{eq:vertex-local-scale}
 |G_ad|_{H^1(\omega_a)}^2\leq C_Vh^3|d|_{\ell^2}^2,
\end{equation}
where $C_V$ is independent of $a,h,$ and $N$.
\end{lemma}

\begin{proof}
Let $\phi_x$ denote the continuous piecewise affine nodal function associated with a mesh vertex $x$. For an active edge $e=[a,b]$ and a vector $s_e\in\R^3$, we set
\begin{equation}\label{eq:vertex-raw-bubble}
 W_{a,b}(s_e)=\phi_a^2\phi_b\,s_e, \qquad W_a(s)=\sum_{[a,b]\in\mathcal E^\circ(a)}W_{a,b}(s_{ab}).
\end{equation}
The first field is a continuous piecewise cubic supported on the edge star of $[a,b]$. It completely vanishes on $\partial\omega_a$, because a face not containing $a$ annihilates $\phi_a$, while a physical boundary face through $a$ annihilates $\phi_b$ since the active edge $[a,b]$ is not contained in that face. On $T=[a,b,c,d]$, the gradient is
\begin{equation}
 \nabla(\phi_a^2\phi_b) = 2\phi_a\phi_b\nabla\lambda_{T,a} +\phi_a^2\nabla\lambda_{T,b}.
\end{equation}
Consequently we achieve
\begin{equation}\label{eq:vertex-raw-properties}
 D_aW_a(s)=A_as,\qquad B_aW_a(s)=0.
\end{equation}
This identity robustly holds on the full edge-jet space, and no independence of the columns of $A_a$ is required.

It remains to dynamically remove the element means. Every interior face $F$ of the vertex star contains $a$. If $F=[x_0,x_1,x_2]$, we let $\beta_F=\phi_{x_0}\phi_{x_1}\phi_{x_2}$. The face bubble is supported on the two tetrahedra $T_F^+$ and $T_F^-$ sharing $F$. Both $\beta_F$ and $\nabla\beta_F$ vanish at every tetrahedron vertex. We orient a unit normal $n_F$ from $T_F^+$ to $T_F^-$ and define
\begin{equation}\label{eq:vertex-face-transfer}
 Z_F=\frac{\beta_F n_F}{\int_F\beta_F}.
\end{equation}
The divergence theorem definitively yields $\int_{T_F^+}\diver Z_F=1$ and $\int_{T_F^-}\diver Z_F=-1$, and all other element means vanish. Moreover, $B_aZ_F=D_aZ_F=0$.

The dual graph of $\T(a)$, considering adjacency across interior faces, is connected. This logically follows directly from the coordinate chain description where adjacent transpositions connect the chains in one cube, and the chains on the two sides of a cube face through $a$ share a tetrahedron face. We fix a spanning tree, and if we establish $m_T(s)=\int_T\diver W_a(s)$, then $\sum_Tm_T(s)=0$ by the divergence theorem and the zero trace of $W_a(s)$ on $\partial\omega_a$. We sequentially route these means from the leaves of the tree to its root using the transfers in \eqref{eq:vertex-face-transfer}. The result is a linear correction $C_a(s)$ satisfying
\begin{equation}
 M_aC_a(s)=-M_aW_a(s),\qquad B_aC_a(s)=D_aC_a(s)=0.
\end{equation}
Thus the combined term $H_a=W_a+C_a$ correctly obeys
\begin{equation}\label{eq:vertex-edgejet-lift}
 D_aH_as=A_as,\qquad B_aH_as=0,\qquad M_aH_as=0.
\end{equation}
This establishes $\operatorname{im}A_a\subset D_a(\ker M_a\cap\ker B_a)$. The reverse inclusion follows strictly from Lemma \ref{lem:vertex-jets}, because every field in $S_a$ extends by zero to a global conforming cubic.

For $d\in\operatorname{im}A_a$, let $J_ad$ be the inverse of the restriction of $A_a$ to $(\ker A_a)^\perp$, and subsequently define
\begin{equation}\label{eq:vertex-explicit-operator}
 G_ad=H_aJ_ad.
\end{equation}
This choice is fundamentally linear. On an $h$-scaled star, $A_a$ scales as $h^{-1}$, which means $|J_ad|\leq Ch|d|$. The raw fields satisfy $|W_a(J_ad)|_{H^1(\omega_a)}^2 \leq Ch|J_ad|^2\leq Ch^3|d|^2$. Their element means are bounded strictly by $Ch^3|d|$. A transfer normalized as in \eqref{eq:vertex-face-transfer} possesses a squared $H^1$ seminorm bounded by $Ch^{-3}$. The tree has at most $23$ edges, and the six vertex states previously outlined in Lemma \ref{lem:vertex-coverage} supply one common bound for $J_a$. The routed correction therefore inherently satisfies the same $Ch^3|d|^2$ estimate, which rigorously proves \eqref{eq:vertex-local-scale}.
\end{proof}

\begin{proposition}[Uniformly stable vertex lift]\label{prop:vertex}
Fix $k\in\{4,5\}$. If $r\in Q_{h,k}$ and $\int_T r=0$ for every $T\in\T_h$, then one can effectively choose $L_h^1r\in V_{h,3}\subset V_{h,k}$, depending linearly on $r$, such that
\begin{equation}\label{eq:vertex-match}
 (\diver L_h^1r)|_T(a)=r|_T(a)
\end{equation}
for every tetrahedron and vertex incidence $(T,a)$. Additionally, this operator satisfies
\begin{equation}\label{eq:vertex-means}
 \int_T\diver L_h^1r=0 \quad (T\in\T_h), \qquad |L_h^1r|_1\leq C_1(k)\norm r,
\end{equation}
where the stability constant $C_1(k)$ is independent of $h$ and $N$.
\end{proposition}

\begin{proof}
For every vertex $a$, we systematically form $d_a=(r|_T(a))_{T\in\T(a)}$. Since $r=\diver w_h$ for a certain $w_h\in V_{h,k}$, Lemma \ref{lem:vertex-jets} indicates that $d_a\in\operatorname{im}A_a$. For each oriented vertex star, we fix once and for all the operator $G_a$ shown in \eqref{eq:vertex-explicit-operator}, and define the total lift as
\begin{equation}
 L_h^1r=\sum_aG_ad_a.
\end{equation}
On a fixed incidence $(T,a)$, the specific summand indexed by $a$ attains the required value $r|_T(a)$. The other three vertex-star fields supported on $T$ logically contain that exact incidence among their $B$-rows and therefore contribute zero, completely proving \eqref{eq:vertex-match}. Each individual summand has a zero divergence mean on every tetrahedron in its support, seamlessly proving \eqref{eq:vertex-means}.

On each distinct tetrahedron, only the four unique fields associated with its vertices can be nonzero. By combining the fixed-degree inverse estimate and \eqref{eq:vertex-local-scale}, we deduce
\begin{equation}
 |L_h^1r|_1^2 \leq4\sum_a|G_ad_a|_1^2 \leq C h^3\sum_a\sum_{T\ni a}|r|_T(a)|^2 \leq C(k)\sum_T\norm{r}_{L^2(T)}^2.
\end{equation}
This bound firmly establishes \eqref{eq:vertex-means}. Since the specific maps $G_a$ are fixed and all continuous evaluations are linear, the total map $L_h^1$ is definitively linear.
\end{proof}

\begin{remark}[Comparison with Zhang]\label{rem:zhang-vertex-cross-check}
Zhang's Lemma 3.2 in \cite[pp.~673-679, (3.2)-(3.19)]{Zhang2011} treats the same six geometric vertex star sizes and uses cubic local velocities. The specific edge-jet characterization formalized in Lemma \ref{lem:vertex-jets} gives the compatibility space directly, while Lemma \ref{lem:vertex-local} physically realizes it by deploying the unified skeleton-bubble calculus defined in Lemma \ref{lem:skeleton-bubble}.
\end{remark}

\section{Analytic Edge-Star Lifts in Degrees Four and Five}\label{sec:raw-edge}

In this section, we develop an analytic edge-star lift for polynomial degrees four and five. For each geometric edge, we construct a local field whose divergence realizes the prescribed data on that edge while vanishing at every other vertex and edge node. The support of this field is strictly contained within a uniformly bounded enlargement of the edge star. Consequently, these local fields can be assembled globally without requiring a global ordering or shared correction degrees of freedom. 

Let \(e\) be a geometric edge of the mesh \(\T_h\). We define its tetrahedron star as
\begin{equation}\label{eq:edge-star}
 \omega_e=\bigcup\{T\in\T_h:e\subset T\}.
\end{equation}
Furthermore, let \(P_e\) be the union of \(\omega_e\) and every tetrahedron sharing a face with a tetrahedron of \(\omega_e\). Thus, \(P_e\) serves as a one-face-neighbor enlargement. Its volumetric size is bounded independently of the mesh parameters \(h\) and \(N\). The support hierarchy and the macroelement geometry utilized in the subsequent proofs are illustrated schematically in Figure~\ref{fig:edge-geometry}.

\begin{figure}[t]
\centering
\begin{tikzpicture}[font=\small,>=Latex]
  \begin{scope}[xshift=0cm]
    \foreach \a in {0,60,...,300}{
      \pgfmathsetmacro{\b}{\a+48}
      \path[fill=blue!28,draw=blue!55] (0,0)--(\a:1.18)--(\b:1.18)--cycle;
    }
    \fill[red!75!black] (0,0) circle (2.2pt);
    \node[align=center] at (0,-1.62) {(a) $\omega_e$\\edge viewed end-on};
    \node[red!60!black,anchor=west] at (.10,.03) {$e$};
  \end{scope}
  \begin{scope}[xshift=4.25cm]
    \foreach \a in {0,60,...,300}{
      \pgfmathsetmacro{\b}{\a+50}
      \path[fill=gray!16,draw=gray!55]
        (\a:1.75)--(\b:1.75)--(\b:1.10)--(\a:1.10)--cycle;
      \path[fill=blue!28,draw=blue!55] (0,0)--(\a:1.10)--(\b:1.10)--cycle;
    }
    \fill[red!75!black] (0,0) circle (2.2pt);
    \draw[->] (1.35,1.12)--(1.68,1.48)
      node[above right,align=left] {$P_e\setminus\omega_e$\\face-neighbor ring};
    \node[align=center] at (0,-2.05) {(b) supported patch $P_e$};
  \end{scope}
  \begin{scope}[xshift=9.6cm,yshift=-.25cm]
    \path[fill=green!10,draw=green!45!black] (0,0) rectangle (1.45,1.45);
    \path[fill=green!10,draw=green!45!black] (1.45,0) rectangle (2.90,1.45);
    \draw[gray!70] (0,0)--(1.45,1.45) (1.45,0)--(2.90,1.45);
    \fill (0.47,.45) circle (2pt) node[below] {$T_C$};
    \fill (2.42,1.00) circle (2pt) node[above] {$T_{C'}$};
    \draw[very thick,->,orange!80!black] (.65,.67) .. controls (1.25,1.25) and
      (1.70,.25) .. (2.24,.80);
    \node[align=center] at (1.45,-.47) {(c) two-cube macro\\mean transfer};
  \end{scope}
\end{tikzpicture}
\caption{Geometry of the edge construction. Panels (a) and (b) are schematic normal cross-sections. The dark sectors are precisely the tetrahedra incident to $e$, and the light ring denotes tetrahedra sharing a face with that star. Panel (c) shows the two face-adjacent cubes used later to route element means. Each cube contains six Freudenthal tetrahedra. The diagrams indicate only containment and support, while the exact combinatorics are defined in the text and generated from integer coordinates.}\label{fig:edge-geometry}
\end{figure}
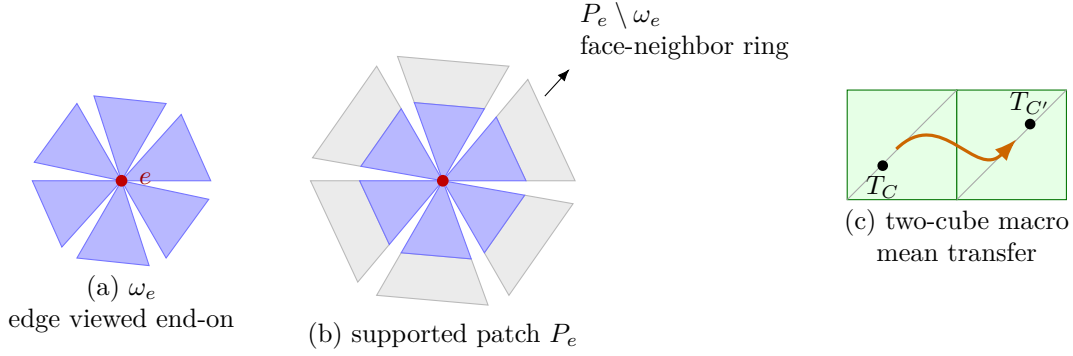

We fix the polynomial degree \(k\in\{4,5\}\). On each broken \(P_{k-1}\) edge trace, we employ the equispaced Lagrange nodes. In this context, a skeleton node is defined as either a tetrahedron vertex or a nonvertex node situated in the relative interior of a tetrahedron edge. The tetrahedron incidences are retained separately to ensure continuity. We formally define two finite element spaces for our localized construction:
\begin{equation}\label{eq:edge-spaces}
\begin{aligned}
 X_{e,k}&=\{v\in[C^0P_k(\omega_e)]^3: v=0\text{ on the portions of }\partial\omega_e\cap\partial\Omega\},\\
 Y_{e,k}&=\{v\in[C^0P_k(P_e)]^3:v|_{\partial P_e}=0\}.
\end{aligned}
\end{equation}
The artificial boundary of \(\omega_e\) is left free in \(X_{e,k}\). Hence, the restriction to \(\omega_e\) of every global velocity field naturally belongs to \(X_{e,k}\). In contrast, every member of \(Y_{e,k}\) extends by zero to a globally conforming member of \(V_{h,k}\).

To manipulate the divergence values across these discrete structures, let \(\tau_e v\) denote the vector of all broken divergence values evaluated at the \((k-2)\) nonvertex pressure nodes of \(e\), allocating one copy for every incident tetrahedron. Let \(\gamma_e v\) be the vector of all broken divergence values at the vertices of the tetrahedra residing in \(\omega_e\). We define the admissible source trace space as
\begin{equation}\label{eq:edge-source-space}
 \mathcal D_{e,k}=\tau_e(\ker\gamma_e\cap X_{e,k}).
\end{equation}
Finally, let \(\sigma_e v\) represent the collection of the divergence values at every skeleton node of \(P_e\), strictly excluding the nonvertex nodes carried by \(e\). 

\subsection{Edge-Star Geometry and Finite Coverage}

The analytic construction formulated in Lemma~\ref{lem:edge-star} relies exclusively on seven distinct incidence types and the coordinate-chain classification. To index the exact verification records provided in the supplement and to make their finite-to-global scope fully explicit, we define finer geometrical states.

\begin{definition}[Local states and record levels]\label{def:local-state}
We rescale the domain by $h^{-1}$, ensuring that all mesh vertices possess integer coordinates within $[0,N]^3$. The conservative coverage state $S_N(e)$ is composed of the oriented edge $e$, the vertex sets of all tetrahedra contained in $\omega_e$ and $P_e$, and the comprehensive list of physical coordinate planes $\{x_i=0,N\}$ intersecting $P_e$. States are considered equivalent when an integer translation followed by a coordinate permutation accurately maps every listed object to the corresponding target object, notably without taking quotients by reflections. A source configuration omits $P_e$ and retains only the physical planes meeting $\omega_e$. Finally, a degree record is identified as one concrete triple $(N,e,k)$, constructed without any symmetry quotient.
\end{definition}

\begin{table}[ht]
\centering
\small
\caption{Exact enumeration of local states and degree records across scaled grids.}\label{tab:enumeration}
\begin{tabular}{@{}lrrrrr@{}}
\toprule
level & quotient & $N=3$ & $N=4$ & $N=5$ & union/total\\
\midrule
geometric incidence type & coarse geometry &7&7&7&7\\
source configuration & translation, axis permutation &64&85&85&85\\
conservative coverage state & translation, axis permutation &64&129&165&180\\
geometric edge placement & none &279&604&1115&1998\\
degree record $(N,e,k)$ & none &558&1208&2230&3996\\
\bottomrule
\end{tabular}
\end{table}

These distinct notions describe hierarchal levels of classification. Their exact enumeration is systematically summarized in Table~\ref{tab:enumeration}. Accordingly, the numbers 7, 85, and 3996 refer respectively to coarse geometric incidence types, inequivalent source configurations, and individual degree records. The sum $1766=2(279+604)$ counts the degree records strictly on the $N=3,4$ grids and does not represent a quotient count. For this particular mesh, the one-ring geometry is uniquely determined by the source state. Moreover, whether the exterior boundary of $P_e$ is physical or artificial does not functionally affect $Y_{e,k}$, because its trace vanishes on the entire patch boundary by definition. Omitting this exterior-boundary label therefore yields 85 inequivalent rational matrix problems. Retaining the label produces 180 conservative coverage states, which serves as the finer classification necessary for the finite-to-infinite extension argument detailed below.

\begin{lemma}[Finite coverage of arbitrary grids]\label{lem:finite-coverage}
Every conservative coverage state $S_N(e)$ with $N\geq3$ is rigorously represented by an edge placement on one of the $N=3,4,5$ grids. Consequently, the exact verification of a local identity on all possible edge placements of those limited grids universally verifies that identity for every $N\geq3$.
\end{lemma}

\begin{proof}
Let $a_i\leq b_i$ be the minimum and maximum $i$-th integer coordinates of the vertices contained in $P_e$. We first establish the uniform spatial bound
\begin{equation}\label{eq:patch-span}
 b_i-a_i\leq3\qquad(i=1,2,3).
\end{equation}
Direct analytical enumeration confirms this sharp bound. The entries detailed in Table~\ref{tab:maximal-span} follow directly from the mathematical definition of the mesh. 

\begin{table}[ht]
\centering
\small
\caption{Maximal spans of edge patches and tetrahedra by geometric incidence type.}\label{tab:maximal-span}
\begin{tabular}{@{}lccc@{}}
\toprule
edge incidence type & maximal span of $\omega_e$ & maximal span of $P_e$
& $\min$--$\max\ |P_e|$\\
\midrule
one-tetrahedron box edge &(1,1,1)&(1,1,1)&3\\
two-tetrahedron box edge &(1,1,1)&(1,1,2)&4--6\\
boundary-face diagonal &(1,1,1)&(1,2,2)&4--6\\
boundary-face axis &(1,1,2)&(1,2,2)&7--9\\
interior face diagonal &(1,1,2)&(2,3,3)&10--12\\
interior coordinate edge &(1,2,2)&(2,2,3)&16--18\\
cube body diagonal &(1,1,1)&(3,3,3)&12--18\\
\bottomrule
\end{tabular}
\end{table}

We can write a generic tetrahedron as the four-vertex sequential chain
\begin{equation}
 K(c,\sigma)= \{c, c+e_{\sigma(1)}, c+e_{\sigma(1)}+e_{\sigma(2)}, c+\one\}, \qquad \sigma\in S_3.
\end{equation}
For the specific coordinate increments $(1,0,0)$, $(1,1,0)$, and $(1,1,1)$, we logically list the six permutations of $\sigma$ whose chain encompasses the two desired endpoints. These yield respectively the coordinate-edge, face-diagonal, and body-diagonal star rows in our classification. For each recorded chain, we sequentially delete each of its four vertices. The unique second tetrahedron attached to the resulting face, provided it exists, forms the topological one-ring. By extracting the coordinate minima and maxima from this construction, we obtain the last two columns of Table~\ref{tab:maximal-span}. No further cases exist, as every valid tetrahedron corresponds to one of the six chains and every interior triangular face corresponds to one of the four vertex deletions. Intersecting these three distinct interior lists with one or two box boundary planes yields precisely the first four boundary rows and geometrically can only decrease the coordinate span. In particular, the final column conclusively proves the bound in \eqref{eq:patch-span}. In the critical body-diagonal case, the edge star occupies the central cube $c$, while its exterior face neighbors occupy the six adjacent cubes $c\pm e_i$. Their combined union therefore dictates a maximal span of $(3,3,3)$.

For grid dimensions $N=3,4$, the specific state naturally occurs in the corresponding enumerated grid. Now suppose $N\geq5$. By \eqref{eq:patch-span}, the interval $[a_i,b_i]$ cannot intersect both $0$ and $N$ simultaneously. We therefore translate the $i$-th coordinate interval using the mapping
\begin{equation}
 [a_i,b_i]\longmapsto
 \begin{cases}
 [0,b_i-a_i],&a_i=0,\\
 [5-(b_i-a_i),5],&b_i=N,\\
 [1,1+(b_i-a_i)],&0<a_i\leq b_i<N.
 \end{cases}
\end{equation}
This deterministic image lies strictly within $[0,5]$, and it intersects the low boundary plane, the high boundary plane, or neither plane exactly when the original given interval does. Applying these three integer translations simultaneously maps $e$, every star tetrahedron, and every one-ring tetrahedron efficiently into the $N=5$ grid. Integer translations structurally preserve the coordinate-order Freudenthal subdivision. Furthermore, this mapping preserves precisely the physical boundary planes recorded in $S_N(e)$, and hence preserves all source nodes deleted by the homogeneous boundary conditions. The space $Y_{e,k}$ remains zero on the entire patch exterior in both realizations. The local incidence matrices are therefore completely identical up to row, column, and vector-coordinate permutations, which finalizes the proof.
\end{proof}

\begin{lemma}[Supported edge-star lift]\label{lem:edge-star}
For every geometric Freudenthal edge \(e\) and for polynomial degrees \(k=4,5\), the localized trace mapping is surjective such that
\begin{equation}\label{eq:edge-star-surj}
 \tau_e(\ker\sigma_e\cap Y_{e,k})=\mathcal D_{e,k}.
\end{equation}
There exists a linear right inverse $G_{e,k}$ defined on this space. On an \(h\)-scaled patch, this inverse operator satisfies the stability bound
\begin{equation}\label{eq:edge-star-scale}
 |G_{e,k}d|_{H^1(P_e)}^2\leq C_E(k)h^3|d|_{\ell^2}^2, \qquad d\in\mathcal D_{e,k},
\end{equation}
where the stability constant is independent of the choice of \(e\), \(h\), and \(N\).
\end{lemma}

\begin{proof}
We introduce the intermediate subspace
\begin{equation}
 \mathcal I_{e,k}=\tau_e(\ker\sigma_e\cap Y_{e,k}).
\end{equation}
Direct restriction from the extended patch $P_e$ to the core star $\omega_e$ guarantees the inclusion
\begin{equation}\label{eq:edge-inclusion}
 \mathcal I_{e,k}\subset\mathcal D_{e,k}.
\end{equation}
This relation holds because a supported velocity field can be extended by zero, and every broken vertex condition is inherently contained among the rows defining $\sigma_e$.

There are exactly seven geometric edge incidence types. We consistently orient the chosen edge from its coordinatewise smaller endpoint $a$ to its larger endpoint $b$. Its scaled coordinate increment is, up to permutation, strictly $(1,0,0)$, $(1,1,0)$, or $(1,1,1)$. Indeed, two vertices of a Freudenthal coordinate chain must differ by the sum of one, two, or three distinct coordinate basis vectors. A body diagonal belongs to the six tetrahedra comprising one unit cube. A face diagonal possesses four incident tetrahedra when its grid-square face is strictly interior and exactly two when that face aligns with the physical boundary. A coordinate edge generates six internal sectors in the bulk interior, three sectors on one physical face, and either one or two sectors at the geometric intersection of two physical faces, contingent on the orientation of the surrounding chains. These possibilities are mutually disjoint and exhaustive. The respective dimensions and topological properties are classified in Table~\ref{tab:edge-incidence-classification}. A source relation listed in the final column is enforced pointwise along the respective edge.

\begin{table}[ht]
\centering
\small
\caption{Edge incidence types, geometric locations, and source space dimensions.}\label{tab:edge-incidence-classification}
\begin{tabularx}{\textwidth}{@{}lclrr>{\raggedright\arraybackslash}X@{}}
\toprule
edge type & increment & location & $n_e$ & $\dim\mathcal D_{e,4}/
\dim\mathcal D_{e,5}$ & source relation\\
\midrule
one-tetrahedron box edge &(1,0,0)&two box faces, one sector&1&0/0&zero\\
two-tetrahedron box edge &(1,0,0)&two box faces, two sectors&2&4/6&none\\
boundary-face diagonal &(1,1,0)&one box face&2&2/3&equal pair\\
boundary-face axis &(1,0,0)&one box face&3&6/9&none\\
interior face diagonal &(1,1,0)&interior&4&6/9&alternating sum\\
interior coordinate edge &(1,0,0)&interior&6&12/18&none\\
cube body diagonal &(1,1,1)&one cube&6&12/18&none\\
\bottomrule
\end{tabularx}
\end{table}

Let $\mathcal C_e$ define the pointwise incidence space specified in the final column of Table~\ref{tab:edge-incidence-classification}, corresponding to the zero vector, the equal-pair line, the checkerboard hyperplane, or the full unconstrained space. The three physical restrictions follow directly from finite element conformity. On a one-tetrahedron box edge, each velocity vector component contains the equations of the two bounding physical faces as polynomial factors. Consequently, its full gradient completely vanishes on their linear intersection. For a boundary-face diagonal, we can appropriately choose coordinates so that the physical face lies on the plane $\{x_1=0\}$. The two corresponding tangential derivatives must vanish. Thus, on either incident tetrahedron, the velocity gradient rigorously assumes the rank-one form $z^\pm\otimes e_1$. The normal direction $e_1$ is tangent to the interior face shared by the two adjacent tetrahedra. Imposing functional continuity across that internal face necessitates $z^+=z^-$. Taking the trace of this gradient therefore forces equal pointwise divergence values.

For an interior face diagonal, the four adjacent tetrahedra act as the sectors cut out by two distinct face planes possessing independent normals $n_1, n_2$. If $G_{\epsilon_1\epsilon_2}$ characterizes the velocity gradient inside a particular sector, the structured summation
\begin{equation}
 G_{++}-G_{-+}-G_{+-}+G_{--}
\end{equation}
is simultaneously expressible in the forms $x\otimes n_1$ and $y\otimes n_2$. Gradient jumps evaluated across any face mathematically annihilate all tangential vector directions. Since the normals $n_1$ and $n_2$ are linearly independent, this matrix sum evaluates identically to zero. Taking its trace directly produces the pointwise relation along the edge,
\begin{equation}\label{eq:edge-checkerboard}
 q_{++}-q_{-+}-q_{+-}+q_{--}=0.
\end{equation}
Thus we confirm the structural containment
\begin{equation}\label{eq:edge-source-contained}
 \mathcal D_{e,k}\subset \mathcal C_e\otimes\lambda_a\lambda_bP_{k-3}(e).
\end{equation}

For subsequent verification of the explicit nodal patterns, we systematically record the barycentric gradients of a standard chain tetrahedron. Using unit lattice coordinates shifted to $\xi=x-c$, the barycentric coordinates of $K(c,\sigma)$ presented in the defined vertex order are exactly
\begin{equation}\label{eq:chain-barycentric}
 \lambda_0=1-\xi_{\sigma(1)},\quad \lambda_1=\xi_{\sigma(1)}-\xi_{\sigma(2)},\quad \lambda_2=\xi_{\sigma(2)}-\xi_{\sigma(3)},\quad \lambda_3=\xi_{\sigma(3)}.
\end{equation}
Consequently, their constant geometric gradients are respectively evaluated as
\begin{equation}\label{eq:chain-gradients}
 -e_{\sigma(1)},\quad e_{\sigma(1)}-e_{\sigma(2)},\quad e_{\sigma(2)}-e_{\sigma(3)},\quad e_{\sigma(3)}.
\end{equation}

We mathematically prove the converse statement by utilizing interior face bubbles. If $F=[a,u,v]$ denotes an interior face, we explicitly define the localized vector field
\begin{equation}\label{eq:endpoint-face-bubble}
 \Psi_{F,a}^{(k)}(z)|_T =\lambda_{T,a}^{k-2}\lambda_{T,u}\lambda_{T,v}z \quad(T\supset F), \qquad \Psi_{F,a}^{(k)}(z)=0\quad\text{otherwise}.
\end{equation}
The two contiguous traces perfectly agree on $F$, and the field definitively vanishes on all other boundary faces of its two-tetrahedron support. Evaluated on the two constituent face edges issuing from the point $a$, we compute
\begin{equation}\label{eq:endpoint-bubble-first}
 \diver\Psi_{F,a}^{(k)}(z)|_T|_{[a,u]} =\lambda_a^{k-2}\lambda_u\,z\cdot\nabla\lambda_{T,v},
\end{equation}
\begin{equation}\label{eq:endpoint-bubble-second}
 \diver\Psi_{F,a}^{(k)}(z)|_T|_{[a,v]} =\lambda_a^{k-2}\lambda_v\,z\cdot\nabla\lambda_{T,u}.
\end{equation}
Its divergence trace strictly vanishes on every other physical mesh edge. These foundational identities emerge immediately from the standard product rule.

For a selected target face $F_c=[a,b,c]$, we designate $T_c^\pm$ as its two adjoining tetrahedra. If the chosen vector $z$ structurally enforces
\begin{equation}\label{eq:spill-annihilation}
 z\cdot\nabla\lambda_{T_c^+,b} =z\cdot\nabla\lambda_{T_c^-,b}=0,
\end{equation}
then the derivative response in \eqref{eq:endpoint-bubble-second} identically vanishes on the spill edge $[a,c]$. The remaining active target incidence vector is formalized as
\begin{equation}\label{eq:target-pattern}
 p(c,z)_T=
 \begin{cases}
 z\cdot\nabla\lambda_{T,c},&T=T_c^+\text{ or }T_c^-,\\
 0,&\text{otherwise}.
 \end{cases}
\end{equation}
The resulting mapping operator $(c,z)\mapsto p(c,z)$ is perfectly independent of the degree $k$.

Table~\ref{tab:spanning-patterns} explicitly lists the spanning patterns for the six nonzero incidence types. In this concise notation, coordinates are abbreviated by three binary digits, the standard physical coordinate vectors are denoted by $e_x,e_y,e_z$, $e_i$ in a pattern represents the $i$-th incidence coordinate vector, and the summation vector is $\mathbf s=e_x+e_y+e_z$. The entries of \(p\) follow the tetrahedron ordering in Table~\ref{tab:incident-tetrahedra}. An entry formulated as $(c;z)\mapsto p$ concisely denotes the evaluation of \eqref{eq:endpoint-face-bubble} on the specific target face $F_c=[a,b,c]$.

\begin{table}[ht]
\centering
\scriptsize
\setlength{\tabcolsep}{3pt}
\caption{Spanning target patterns for the six nonzero incidence configurations.}\label{tab:spanning-patterns}
\begin{tabularx}{\textwidth}{@{}lclX@{}}
\toprule
type & endpoint & $e$ & spanning patterns $(c;z)\mapsto p$\\
\midrule
boundary-face axis
&$a$&$(001,011)$&
$(111;e_x+e_y)\mapsto(1,1,0)$;
$(111;e_z)\mapsto(0,-1,0)$;
$(112;\mathbf s)\mapsto(0,1,1)$\\
&$b$&&
$(111;e_x)\mapsto(1,1,0)$;
$(112;e_x)\mapsto(0,0,1)$;
$(112;e_z)\mapsto(0,1,0)$\\[1mm]
boundary-face diagonal
&$a$&$(000,011)$&$(111;\mathbf s)\mapsto(1,1)$\\
&$b$&&$(111;e_x)\mapsto(1,1)$\\[1mm]
cube diagonal
&$a$&$(000,111)$&
$(011;e_y)\mapsto e_5$; $(011;e_z)\mapsto e_3$;
$(101;e_x)\mapsto e_4$; $(101;e_z)\mapsto e_1$;
$(110;e_x)\mapsto e_2$; $(110;e_y)\mapsto e_0$\\
&$b$&&
$(001;e_x)\mapsto-e_4$; $(001;e_y)\mapsto-e_5$;
$(010;e_x)\mapsto-e_2$; $(010;e_z)\mapsto-e_3$;
$(100;e_y)\mapsto-e_0$; $(100;e_z)\mapsto-e_1$\\[1mm]
interior axis
&$a$&$(011,111)$&
$(000;e_y)\mapsto-e_0$; $(000;e_z)\mapsto-e_1$;
$(112;e_y)\mapsto-e_5$; $(112;e_x+e_z)\mapsto e_2+e_5$;
$(121;e_z)\mapsto-e_4$; $(121;e_x+e_y)\mapsto e_3+e_4$\\
&$b$&&
$(001;e_z)\mapsto e_1$; $(001;e_x+e_y)\mapsto-e_1-e_2$;
$(010;e_y)\mapsto e_0$; $(010;e_x+e_z)\mapsto-e_0-e_3$;
$(122;e_y)\mapsto e_5$; $(122;e_z)\mapsto e_4$\\[1mm]
interior face diagonal
&$a$&$(001,111)$&
$(000;e_z)\mapsto(-1,-1,0,0)$;
$(011;e_y)\mapsto(0,1,0,1)$;
$(101;e_x)\mapsto(1,0,1,0)$\\
&$b$&&
$(000;\mathbf s)\mapsto(-1,-1,0,0)$;
$(011;e_x)\mapsto(0,-1,0,-1)$;
$(101;e_y)\mapsto(-1,0,-1,0)$\\[1mm]
two-tetrahedron box edge
&$b$&$(000,001)$&
$(111;e_x)\mapsto(0,1)$;
$(111;e_y)\mapsto(1,0)$\\
\bottomrule
\end{tabularx}
\end{table}

\begin{table}[ht]
\centering
\scriptsize
\caption{Ordering of the tetrahedra incident to the six nonzero edge types.}
\label{tab:incident-tetrahedra}
\begin{tabularx}{\textwidth}{@{}lclX@{}}
\toprule
type & $e$ & $n_e$ & ordered pairs $T_i\setminus e$\\
\midrule
boundary-face axis&$(001,011)$&3&
$(000,111),(111,112),(012,112)$\\
boundary-face diagonal&$(000,011)$&2&
$(010,111),(001,111)$\\
cube diagonal&$(000,111)$&6&
$(100,110),(100,101),(010,110),(010,011),(001,101),(001,011)$\\
interior axis&$(011,111)$&6&
$(000,010),(000,001),(001,112),(010,121),(121,122),(112,122)$\\
interior face diagonal&$(001,111)$&4&
$(000,101),(000,011),(101,112),(011,112)$\\
two-tetrahedron box edge&$(000,001)$&2&
$(101,111),(011,111)$\\
\bottomrule
\end{tabularx}
\end{table}

Equations \eqref{eq:chain-barycentric} through \eqref{eq:chain-gradients} demonstrate directly that every vector $z$ listed in Table~\ref{tab:spanning-patterns} robustly satisfies the annihilation constraint in \eqref{eq:spill-annihilation} and precisely yields the stated output pattern. For example, regarding the boundary-face diagonal $e=(000,011)$, the two connected tetrahedra are $[000,010,011,111]$ and $[000,001,011,111]$. Evaluated at the endpoint $000$, setting $c=111$ and $z=\mathbf s$ entirely annihilates the local gradients of $\lambda_{011}$ in both corresponding tetrahedra, while confirming $z\cdot\nabla\lambda_{111}=1$ on each element. This analytical process robustly returns the equal-pair pattern $(1,1)$. The identical structural calculation applies for the endpoint $011$ using the vector $z=e_x$. The cube-diagonal patterns yield signed coordinate unit vectors. The interior-axis patterns organically form three paired triangular blocks. The boundary-face-diagonal pattern successfully spans the complete equal-pair line. Furthermore, the three interior-face-diagonal patterns are mathematically independent and fully span the hyperplane defined by \eqref{eq:edge-checkerboard}. Hence, this verified configuration accurately spans $\mathcal C_e$ at both geometrical endpoints, with a single notable exception.

To address the missing structural endpoint $a=000$ of the two-tetrahedron box edge, we formally set $b=001$, $c=111$, and $u=011$. On the two corresponding indicated face patches, we construct the specialized formulas
\begin{equation}\label{eq:two-tet-borrow-zero}
 \Xi_0^{(k)} =\lambda_a^{k-2}\lambda_b\lambda_c e_y +\lambda_a^{k-2}\lambda_u\lambda_c e_y,
\end{equation}
\begin{equation}\label{eq:two-tet-borrow-one}
 \Xi_1^{(k)} =\lambda_a^{k-2}\lambda_b\lambda_c(e_x+e_z) -\lambda_a^{k-2}\lambda_u\lambda_c e_y.
\end{equation}
The first summand in each explicit line is strictly supported on the triangle $[a,b,c]$, while the second is supported on $[a,u,c]$. The latter operates as an inward interior face of the encompassing one-ring $P_e$. Applying equations \eqref{eq:endpoint-bubble-first} through \eqref{eq:endpoint-bubble-second} rigorously yields the target trace patterns $(1,0)$ and $(0,1)$, respectively, completely maintaining a zero divergence trace on every other continuous skeleton edge.

We now lift the polynomial divergence trace. Because it strictly vanishes at both $a$ and $b$, it admits the unique uncoupled factorization
\begin{equation}\label{eq:edge-mode-decomposition}
\begin{aligned}
k=4:\quad & q_T=\alpha_T\lambda_a^2\lambda_b +\beta_T\lambda_a\lambda_b^2,\\
k=5:\quad & q_T=\alpha_T\lambda_a^3\lambda_b +\mu_T\lambda_a^2\lambda_b^2 +\beta_T\lambda_a\lambda_b^3.
\end{aligned}
\end{equation}
By the containment property established in \eqref{eq:edge-source-contained}, each vector coefficient inherently belongs to the space $\mathcal C_e$. We sequentially expand the coefficients $\alpha$ and $\beta$ in the corresponding valid rows of the spanning pattern table and logically sum the resulting endpoint bubbles, explicitly incorporating \eqref{eq:two-tet-borrow-zero} and \eqref{eq:two-tet-borrow-one} when necessary. This precise protocol fully realizes the two boundary endpoint modes and introduces no other parasitic skeleton traces.

To isolate the middle quintic mode when $k=5$, we utilize the following localized function on the identical target face:
\begin{equation}\label{eq:middle-face-bubble}
 \Theta_{F_c}(z)|_T =\lambda_{T,a}^2\lambda_{T,b}^2\lambda_{T,c}z \quad(T\supset F_c),
\end{equation}
which is seamlessly extended by zero. Its spatial divergence perfectly vanishes on every edge except $e$, where it evaluates to
\begin{equation}
 \diver\Theta_{F_c}(z)|_T|_e =\lambda_a^2\lambda_b^2z\cdot\nabla\lambda_{T,c}.
\end{equation}
Thus, any target-face spanning row reliably realizes the parameter $\mu$. In the isolated two-tetrahedron case, one securely uses the verified endpoint-$b$ row. We have therefore proved that
\begin{equation}
 \mathcal C_e\otimes\lambda_a\lambda_bP_{k-3}(e) \subset\mathcal I_{e,k}.
\end{equation}
Coupled with \eqref{eq:edge-inclusion} and \eqref{eq:edge-source-contained}, this comprehensively proves the surjectivity asserted in \eqref{eq:edge-star-surj}.

The enumerated coordinate-chain list also verifies geometric coverage for arbitrary values of $N$. Every physical edge star is securely mapped to one of the canonical displayed stars by applying an integer translation, a coordinate permutation, and possibly a simultaneous coordinate reversal. All constructed bubbles, except for those in \eqref{eq:two-tet-borrow-zero} and \eqref{eq:two-tet-borrow-one}, utilize exclusively interior faces containing $e$ and are therefore totally unaffected by any further physical clipping of $P_e$. In the single exceptional case, the chain list affirmatively identifies $[a,u,c]$ as an inward face sharing a valid star tetrahedron, ensuring its second tetrahedron lies strictly within $P_e$. Thus, no additional patch geometry is dynamically required, and the exact formulas apply uniformly on every global grid.

Finally, conversion between the nodal values $\tau_ev$ and the coefficients in \eqref{eq:edge-mode-decomposition} is uniformly bounded for fixed $k$. The displayed spanning families define fixed linear right inverses on the seven normalized patches. Under $x=x_e+hQ\widehat x$, each reference coefficient vector is multiplied by $hQ$, where $Q$ is the signed coordinate permutation corresponding to the coordinate permutation and possible central inversion. Since $Q$ is orthogonal, the divergence values are preserved and the squared $H^1$ seminorm acquires the factor $h^3$. Taking the maximum over the seven families proves \eqref{eq:edge-star-scale}.
\end{proof}

\begin{lemma}[Global edge lift before mean adjustment]\label{lem:raw-edge}
Let \(k=4\) or \(5\), and suppose the pressure field \(r\in Q_{h,k}\) vanishes identically at every broken geometric vertex. There exist compactly supported fields \(u_e\in V_{h,k}\), specifically allocated one for each geometric edge and depending linearly on \(r\), such that the conditions
\begin{equation}\label{eq:off-target}
 (\diver u_e)|_T(a)=0 \quad\text{at every skeleton node not carried by \(e\)},
\end{equation}
\begin{equation}\label{eq:on-target}
 (\diver u_e)|_T(a)=r|_T(a) \quad\text{at every nonvertex node \(a\) of \(e\)},
\end{equation}
are precisely met. Additionally, the accumulated global fields satisfy the stability bound
\begin{equation}\label{eq:raw-square-sum}
 \sum_e|u_e|_1^2\leq C_{\mathrm{raw}}(k)^2\norm r^2.
\end{equation}
The geometrical edge patches exhibit a uniformly bounded spatial overlap. Therefore, the global assembly \(u=\sum_eu_e\) successfully matches \(r\) identically on every continuous tetrahedron edge.
\end{lemma}

\begin{proof}
We initiate the construction by choosing a global field \(w_h\in V_{h,k}\) defined with \(\diver w_h=r\). Mathematically restricting \(w_h\) to the inner star \(\omega_e\) clearly demonstrates that its localized target vector \(d_e\) naturally belongs to the source trace space \(\mathcal D_{e,k}\), precisely because all broken vertex values of \(r\) vanish by hypothesis. We subsequently set \(u_e=G_{e,k}d_e\) and extend it by zero outside the local patch. The trace conditions formulated in \eqref{eq:off-target} and \eqref{eq:on-target} logically follow directly from the surjectivity established in \eqref{eq:edge-star-surj}. Every structural pressure edge trace inherently possesses a polynomial degree of at most \(k-1\). Since its two terminating endpoint values and its \(k-2\) interior Lagrange values uniquely determine this one-dimensional polynomial, the assembled global divergence perfectly agrees with \(r\) on the entire complete edge geometry.

To secure the global bound, we apply fixed-degree norm equivalence independently on each spatial tetrahedron to yield
\begin{equation}
 h^3\sum_{T\in\T_h}\sum_{a\in\mathcal S_{k-1}(T)} |r|_T(a)|^2\leq C(k)\norm r^2.
\end{equation}
In this context, \(\mathcal S_{k-1}(T)\) represents the discrete set of degree-\((k-1)\) localized Lagrange nodes positioned on the vertices and edges of \(T\). Each distinct functional component of the localized vectors \(d_e\) occurs exactly once in this global sum. Hence, applying \eqref{eq:edge-star-scale} deterministically implies the bound expressed in \eqref{eq:raw-square-sum}. A standardized one-ring edge patch topologically contains at most eighteen distinct tetrahedra. Moreover, if a specific tetrahedron resides in \(P_e\), then \(e\) mathematically must be an edge either of that exact tetrahedron or of one of its at most four immediate face neighbors. Thus, it can lie in at most \(6+4\times 6=30\) distinct overlapping edge patches. This structural geometry conclusively proves the principle of bounded overlap, firmly holding true even adjacent to the physical global boundary.
\end{proof}

\section{Explicit Quartic and Quintic Mean Transfers}\label{sec:mean-transfer}

Having established the local edge lifts, we next address the residual element means. In this section, we introduce explicit quartic and quintic mean transfer operators designed to route these residuals while preserving previous invariant trace conditions.

\subsection{A Quartic Two-Cube Transfer}

We define the normalized reference domain block as
\begin{equation}
 \widehat P=[0,2]\times[0,1]^2.
\end{equation}
This specific domain encapsulates the twelve tetrahedra obtained strictly from the standard Freudenthal geometric split of its two constituent reference cubes. We establish the constrained polynomial space
\begin{equation}
 W_4(\widehat P)=\{w\in[C^0(\widehat P)]^3: w|_T\in P_4(T)^3,\ w|_{\partial\widehat P}=0, (\diver w)|_T|_e=0\ \forall(T,e)\}.
\end{equation}
For any functional component $w\in W_4(\widehat P)$, we define its structural element mean mapping as
\begin{equation}
 \mu(w)=\left(\int_T\diver w\right)_{T\subset\widehat P}\in\R^{12}.
\end{equation}

\begin{lemma}[Quartic two-cube macro lift]\label{lem:macro}
The operational image of the mean transfer exactly covers the zero-sum subspace, fulfilling
\begin{equation}\label{eq:macro-image}
 \mu(W_4(\widehat P))=\one^\perp =\{m\in\R^{12}:\textstyle\sum_Tm_T=0\}.
\end{equation}
There is a fixed linear right inverse operator $B:\one^\perp\to W_4(\widehat P)$. On any translated, coordinate-permuted, and $h$-scaled geometric copy $P_h$ representing $\widehat P$, it logically induces a corresponding right inverse satisfying the localized scale bound
\begin{equation}\label{eq:macro-scale}
 |B_hm|_{H^1(P_h)}\leq C_Mh^{-3/2}|m|_{\ell^2}.
\end{equation}
\end{lemma}

\begin{proof}
We systematically order the constituent tetrahedra first by their host cube, and within each individual cube structurally by the fixed Freudenthal mapping sequence
\begin{equation}
 (xyz,xzy,yxz,yzx,zxy,zyx).
\end{equation}
For an interior quartic Bernstein geometric control point denoted by $p_{ijk}=(i/4,j/4,k/4)$, we let $B_{ijk}$ define the global conforming Bernstein basis function whose Bernstein coefficient associated with $p_{ijk}$ is one and whose remaining structural coefficients are strictly zero. Each function $B_{ijk}$ deployed below possesses a guaranteed zero trace evaluated on the exterior domain boundary $\partial\widehat P$. With $e_x,e_y,e_z$ denoting the standard directional coordinate vectors, we explicitly define the distinct basis fields:
\begin{equation}\label{eq:bernstein-basis-w}
\begin{aligned}
w_1&=-B_{422}e_z+(B_{421}+B_{423})e_y,\\
w_2&= B_{221}e_y-B_{211}e_x+B_{212}e_z,\\
w_3&=(B_{432}+B_{332}+B_{232})e_x,\\
w_4&= B_{221}e_x-B_{121}e_y+B_{122}e_z,\\
w_5&=(B_{423}+B_{323}+B_{223})e_x,\\
w_6&= B_{621}e_y-B_{611}e_x+B_{612}e_z,\\
w_7&=-B_{732}e_z+B_{722}e_y+B_{632}e_x,\\
w_8&=(B_{431}+B_{331}+B_{531})e_x,\\
w_9&=(B_{421}+B_{621}+B_{521})e_x,\\
w_{10}&=B_{722}e_z-B_{723}e_y+B_{623}e_x,\\
w_{11}&=(B_{413}+B_{313}+B_{513})e_x.
\end{aligned}
\end{equation}

We mathematically verify the two required trace properties directly from standard Bernstein calculus operations. Evaluated on a unit Freudenthal reference tetrahedron parameterized as $T=[y_0,y_1,y_2,y_3]$, we can write the local function as $w=\sum_{|\alpha|=4}c_\alpha B_\alpha^4$. If $\varepsilon_i$ serves as the $i$-th directional unit multi-index, then the local divergence is computed as
\begin{equation}\label{eq:bernstein-divergence}
 \diver w =4\sum_{|\beta|=3} \left(\sum_{i=0}^3c_{\beta+\varepsilon_i} \cdot\nabla\lambda_i\right)B_\beta^3.
\end{equation}
The specific Bernstein basis functions whose structural multi-indices contain at most two nonzero numerical entries are precisely those functionally carried by the geometrical tetrahedron edges. The direct substitution of the three distinct nonzero localized coefficients characterizing each generated $w_j$ into \eqref{eq:bernstein-divergence} seamlessly results in
\begin{equation}
 (\diver w_j)|_T|_e=0 \qquad(T\subset\widehat P,\ e\subset T),
\end{equation}
thus confirming that $w_j\in W_4(\widehat P)$.

Since every degree-three Bernstein basis function $B_\beta^3$ on a unit-determinant tetrahedron has integral $1/120$, the differentiation formula gives
\begin{equation}\label{eq:bernstein-mean}
 \int_T\diver w =\frac1{30}\sum_{|\alpha|=4} c_\alpha\cdot \sum_{\{i:\alpha_i>0\}}\nabla\lambda_i.
\end{equation}
Applying the relation in \eqref{eq:bernstein-mean} strictly to the previously displayed fields, we verify that the first eleven components of their resultant mean output vectors correctly form the scaled matrix $A/30$, where
\begin{equation}\label{eq:quartic-mean-matrix}
A=\begin{bmatrix}
-1& 1& 0&-1& 0& 0& 0& 0& 1& 0& 0\\
 1& 1& 0& 0& 0& 0& 0& 0& 0& 0& 0\\
 0&-1& 0& 1& 0& 0& 0& 1& 0& 0& 0\\
 0& 0& 1& 1& 0& 0& 0& 0& 0& 0& 0\\
 0&-1& 0& 0& 0& 0& 0& 0& 0& 0& 1\\
 0& 0& 0&-1& 1& 0& 0& 0& 0& 0& 0\\
 0& 0& 0& 0& 0& 1&-1& 0&-1& 1& 0\\
 0& 0& 0& 0& 0& 1& 1& 0& 0&-1& 0\\
 0& 0& 0& 0& 0&-1&-1&-1& 0& 0& 0\\
-1& 0&-1& 0& 0& 0& 1& 0& 0& 0& 0\\
 0& 0& 0& 0& 0&-1& 0& 0& 0&-1&-1
\end{bmatrix},
\end{equation}
and we can verify that $\det A=-6$. Thus the generated eleven mean operational vectors are definitively linearly independent. Their vector components naturally sum to zero solely because every defined $w_j$ explicitly maintains a continuous zero trace evaluated on $\partial\widehat P$. It immediately follows that they successfully form a complete structural basis of $\one^\perp$, which securely proves the image mapping property \eqref{eq:macro-image}. If $m'$ symbolically denotes the first eleven localized components of any generic vector $m\in\one^\perp$, an explicit operational right inverse is efficiently evaluated as
\begin{equation}
 Bm=30\sum_{j=1}^{11}(A^{-1}m')_j w_j.
\end{equation}

If $Q$ represents a standardized coordinate permutation, $x_0$ denotes the lower coordinate corner of the targeted physical two-cube macro patch, and the generated mean entries are structurally permuted directly with the elements, we define the rescaled mapped field as
\begin{equation}
 B_hm(x)=h^{-2}Q(Bm)\left(Q^T\frac{x-x_0}{h}\right).
\end{equation}
Its resulting operational tetrahedron divergence integrals evaluate perfectly to equal $m$, and its squared $H^1$ analytic seminorm equals exactly $h^{-3}|Bm|_{H^1(\widehat P)}^2$. Continuous finite-dimensional mathematical norm equivalence definitively proves \eqref{eq:macro-scale}. After standardized translation down to the origin corner, any arbitrary two face-neighboring physical cubes are efficiently mapped back to the reference $[0,2]\times[0,1]^2$ strictly by employing a proper sequence of coordinate permutations.
\end{proof}

\subsection{An Explicit Quintic Face Transfer}

\begin{lemma}[Quintic two-tetrahedron transfer]\label{lem:quintic}
If two adjacent Freudenthal tetrahedra $T^+$ and $T^-$ structurally share an interior face, there exists a conforming continuous piecewise-$P_5$ vector field $\psi_F$, completely supported on their structural union, such that
\begin{equation}\label{eq:quintic-properties}
 \int_{T^+}\diver\psi_F=1,\qquad \int_{T^-}\diver\psi_F=-1,\qquad (\diver\psi_F)|_T|_e=0\quad\forall(T,e).
\end{equation}
Its continuous exterior boundary trace is strictly zero, and evaluated on an $h$-scaled element pair it maintains the bound
\begin{equation}\label{eq:quintic-scale}
 |\psi_{F,h}|_1\leq C_5h^{-3/2}
\end{equation}
when mathematically normalized as designated in \eqref{eq:quintic-properties}.
\end{lemma}

\begin{proof}
On the reference canonical element pair defined spatially by
\begin{equation}
 K_H=\{0\leq x\leq y\leq z\leq1\},\qquad K_F=\{0\leq y\leq x\leq z\leq1\},
\end{equation}
we formulate the fixed directional vector $t=(1,-1,0)$ and explicitly construct
\begin{equation}\label{eq:quintic-explicit}
 \psi_H=630t(1-z)(z-y)^2x^2,\qquad \psi_F=630t(1-z)(z-x)^2y^2.
\end{equation}
The continuous functional traces flawlessly agree along the internal boundary $x=y$, and the derived vector field reliably vanishes on all six exterior surrounding faces. Direct spatial differentiation robustly calculates
\begin{equation}
\begin{aligned}
 \diver\psi_H&=1260(1-z)x(z-y)(x+z-y),\\
 \diver\psi_F&=-1260(1-z)y(z-x)(y+z-x).
\end{aligned}
\end{equation}
Each analytical expression inherently vanishes on every individual geometric edge belonging to its host tetrahedron. Exact structural integration computed strictly over the continuous domains $0\leq x\leq y\leq z\leq1$ and $0\leq y\leq x\leq z\leq1$ successfully yields
\begin{equation}\label{eq:quintic-integrals}
 \int_{K_H}\diver\psi_H=1,\qquad \int_{K_F}\diver\psi_F=-1,\qquad |\psi|_1^2=\frac{504}{11}.
\end{equation}

The extended geometric construction operating on an arbitrary  interior element face proceeds as follows. For a specifically selected face edge $E=[a_i,a_j]$ and the two distinct geometric vertices $v^+,v^-$ located exactly opposite the shared common face, we evaluate the determinant value
\begin{equation}
 \Delta(E)=\det(a_j-a_i,\,v^+-a_i,\,v^--a_i).
\end{equation}
The resultant exterior spatial faces $[a_i,a_j,v^+]$ and $[a_i,a_j,v^-]$ mathematically sit exactly coplanar if and only if $\Delta(E)=0$. Up to uniform geometric translation, signed coordinate permutation, and structural central inversion, the shared common face and its two opposite vertices are uniquely classified. Specifically, for the interior common face $F=(000,001,111)$, the opposite vertices are deterministically located at $(011,101)$. For the alternative interior face $F=(001,011,111)$, the geometric opposite vertices are similarly identified at $(000,112)$. For the corresponding three face edges processed securely in lexicographic order, the corresponding computed values of the determinant parameter $\Delta$ evaluate strictly to $(-1,1,0)$ and $(-1,0,-1)$, respectively. Hence, every arbitrary interior face undeniably maintains a uniquely distinguished functional edge, and it evaluates specifically as the structural interval $[001,111]$ in both standard canonical geometric configurations.

We logically write the target common face as $F=[a_0,a_1,a_2]$, designating the unique distinguished edge as $[a_1,a_2]$, and we let $\lambda_i^\pm$ explicitly denote the defined barycentric coordinate of vertex $a_i$ operating in the domain element $T^\pm$. We subsequently choose a constant mathematical vector $t$ geometrically tangent to the derived common plane governing the two opposite faces, but structurally positioned not parallel to the distinguished segment $[a_1,a_2]$. The established common plane and the local plane containing $F$ remain geometrically distinct and strictly intersect precisely traversing the directional line through $[a_1,a_2]$. Consequently $t\cdot n_F\ne0$, and the localized function
\begin{equation}\label{eq:quintic-bary}
 \psi^\pm=c\,t\,\lambda_0^\pm(\lambda_1^\pm)^2(\lambda_2^\pm)^2
\end{equation}
exhibits mathematically matching traces seamlessly across $F$ and maintains a strictly zero exterior domain trace. On the unique distinguished geometric edge, the only term potentially surviving localized differentiation inherently contains the structural factor $t\cdot\nabla\lambda_0^\pm=0$. On the other two accompanying face edges, a corresponding squared spatial factor intrinsically remains unharmed, and on the three connected edges meeting the opposite vertex, at least one of two defined vanishing factors persistently remains intact following localized differentiation. Thus all defined broken edge mathematical traces of the resulting divergence evaluate structurally to zero. Finally, the localized volume integral yields
\begin{equation}
 \int_{T^\pm}\diver\psi^\pm =\pm c(t\cdot n_F)\int_F\mu_0\mu_1^2\mu_2^2,
\end{equation}
where $\mu_i$ designate the specified face barycentric coordinates. Because the resulting final spatial integral evaluates uniformly positive, the constant $c$ effectively normalizes the operational functional means exactly to $(1,-1)$. Linearly scaling a unit-integral local mathematical field robustly by the factor $h^{-2}$ completely satisfies and provides the stability bound in \eqref{eq:quintic-scale}.
\end{proof}

\section{Local Mean Routing and the Mean-Preserving Edge Lift}\label{sec:routing}

The local constructions established previously are combined into a globally stable edge lift by efficiently routing the element means. We let $\mathcal G_N$ be the complete set of the $N^3$ closed grid cubes. Fixing a subcollection $\mathcal C\subset\mathcal G_N$, we define the localized subspace
\begin{equation}\label{eq:local-subspace}
 U_{h,k}(\mathcal C)= \{u\in V_{h,k}:\supp u\subset |\mathcal C|\}, \qquad |\mathcal C|=\bigcup_{C\in\mathcal C}C.
\end{equation}
The corresponding cube intersection graph possesses the vertex set $\mathcal C$, with two distinct cubes considered adjacent when their closed physical cubes share a nonempty intersection. Thus, an edge within this graph may represent a geometric face, an edge, or a vertex contact. For each fixed collection $\mathcal C$, the following lemma constructs a robust linear mean-correction operator operating on $U_{h,k}(\mathcal C)$.

\begin{lemma}[Bounded-cluster mean routing]\label{lem:routing}
For every constant $L\geq1$, there exist numerical bounds $A(L)$ and $C(L)$ satisfying the following property. For every grid parameter $N\geq2$, every polynomial degree $k\in\{4,5\}$, and every fixed collection $\mathcal C\subset\mathcal G_N$ of at most $L$ cubes whose cube intersection graph is connected, two specific structural conditions are met. First, there exists an enlargement $\mathcal C^+\supset\mathcal C$ containing at most $A(L)$ cubes. Second, there is a linear operator $\mathcal R_{\mathcal C,k}:U_{h,k}(\mathcal C)\to V_{h,4}$ that depends strictly on $N$, $k$, and $\mathcal C$. For any vector field $u\in U_{h,k}(\mathcal C)$, we set $c=\mathcal R_{\mathcal C,k}u$ and $m_T=\int_T\diver u$. The resulting correction explicitly satisfies
\begin{equation}\label{eq:routing-properties}
 \supp c\subset|\mathcal C^+|,\qquad (\diver c)|_T|_e=0,\qquad \int_T\diver(u+c)=0\quad(T\in\T_h),\qquad |c|_1\leq C(L)|u|_1.
\end{equation}
Furthermore, both stability constants $A(L)$ and $C(L)$ are strictly independent of $h$, $N$, and the chosen cluster $\mathcal C$.
\end{lemma}

\begin{proof}
All algorithmic choices formulated below depend exclusively on $N$ and $\mathcal C$. We sequentially order the grid cubes lexicographically by their integer lower corners. We then subdivide every edge of a lexicographic spanning tree of the cube intersection graph via the unique path that alters differing coordinates strictly in the sequence $x, y, z$. Because intersecting cubes possess lower corners differing by at most one unit in each coordinate direction, each resulting sequential path requires at most three face steps. All intermediate lower corners remain coordinatewise bounded between the two endpoints and therefore safely belong to the grid. Consequently, at most two intermediate cubes are added for each of the at most $L-1$ tree edges. Their total union $\mathcal H(\mathcal C)$ is therefore a face-connected geometric enlargement comprising at most $3L$ cubes. Within every individual cube, we choose the first of its six coordinate-order tetrahedra to act as the local anchor. We subsequently choose the lexicographically first available face neighbor situated inside $\Omega$ to serve as its companion, a choice securely guaranteed to exist for $N\geq2$. Finally, we select the lexicographic breadth-first tree spanning $\mathcal H(\mathcal C)$. For each oriented adjacent cube pair, we fix an analytical right inverse drawn from \cref{lem:macro}. Together with the preceding geometric choices, these operators depend intrinsically only on $N$ and $\mathcal C$. We designate $\mathcal C^+$ as the comprehensive union of all cubes utilized in this construction. Because at most one companion is supplemented for each cube of $\mathcal H(\mathcal C)$, we guarantee $|\mathcal C^+|\leq6L$, naturally permitting the choice $A(L)=6L$.

For a given field $u\in U_{h,k}(\mathcal C)$, we extend the functional $m_T$ by zero over the tetrahedra contained in $|\mathcal H(\mathcal C)|\setminus|\mathcal C|$. Since $u\in V_{h,k}$ and perfectly vanishes outside $|\mathcal C|$, the divergence theorem analytically dictates $\sum_Tm_T=0$. For every non-anchor tetrahedron belonging to a cube $C$, we prescribe the exact correction mean $\rho_T=-m_T$, while prescribing on its specific anchor the aggregated value
\begin{equation}\label{eq:anchor-prescribe}
 \rho_{T_C}=\sum_{T\subset C,\ T\ne T_C}m_T.
\end{equation}
The six corresponding entries of $\rho$ within $C$ structurally sum to zero. We apply the corresponding macro right inverse operating on $C$ and its companion to the mapped twelve-vector evaluating to $\rho$ on $C$ and evaluating to zero on the companion. After comprehensively executing this for every cube, all non-anchor means of $u$ combined with the correction completely vanish, and the residual accumulated mean at $T_C$ precisely equals $g_C=\sum_{T\subset C}m_T$.

We logically root the fixed continuous face tree. For each individual cube $C$, we define the subtree sum
\begin{equation}\label{eq:subtree-sum}
 G_C=\sum_{D\text{ in the subtree rooted at }C}g_D.
\end{equation}
These generated numbers are fixed linear functions evaluated from the original element means. Commencing systematically at the tree leaves, we apply the macro procedure on each child and parent pair directly to the twelve-vector holding $-G_C$ at the child anchor, holding $+G_C$ at the parent anchor, and maintaining zero elsewhere. After all descendants of $C$ have been appropriately processed, the numerical value resting at its anchor is exactly $G_C$. This targeted operation completely eliminates it and correctly passes the integrated subtree total upwards to the parent. At the terminal root, the remaining unallocated value evaluates to $G_{C_{\rm root}}=\sum_Cg_C=0$. We define the total operator $\mathcal R_{\mathcal C,k}u$ to be the analytical sum of all these localized macro fields.

This procedure requires at most $C_0(L)$ distinct algorithmic operations. Every routed vector behaves as a fixed linear combination of the foundational original means, ensuring that
\begin{equation}\label{eq:routing-coeff}
 \sum_j|\rho^{(j)}|_{\ell^2}^2\leq C(L)\sum_T|m_T|^2.
\end{equation}
The deployed macro supports exhibit a spatial overlap bounded strictly by $C(L)$. Moreover, the individual element variations are bounded by
\begin{equation}\label{eq:mean-variation}
 |m_T|\leq |T|^{1/2}\norm{\diver u}_{L^2(T)} \leq C h^{3/2}|u|_{H^1(T)}.
\end{equation}
Combining this local estimate with \eqref{eq:macro-scale} and \eqref{eq:routing-coeff} conclusively gives $|c|_1\leq C(L)|u|_1$. Every constructed macro possesses a zero patch trace alongside an edge-zero divergence, maintaining this behavior even on physical boundary patches. Therefore, extending the field by zero remains globally conforming and flawlessly preserves any homogeneous boundary data.

The designated functionals $u\mapsto m_T$, the finite routing maps, and their aggregate sum are fundamentally linear in nature. Thus, for each geometrically fixed cluster $\mathcal C$, the specified construction accurately defines a continuous linear operator $\mathcal R_{\mathcal C,k}$ mapping on $U_{h,k}(\mathcal C)$. The resulting operational family $\{\mathcal R_{\mathcal C,k}\}_{\mathcal C}$ is uniformly bounded across all structural clusters of the prescribed size.
\end{proof}

For the specific degree $k=5$, \cref{lem:quintic} effectively furnishes routing directed along a tetrahedron and face spanning tree. The subsequent argument utilizes the quartic macro uniformly for both degrees, exploiting the natural containment $V_{h,4}\subset V_{h,5}$.

\begin{proposition}[Uniform edge lift in degrees four and five]\label{prop:edge}
Let $k\in\{4,5\}$ and assume a grid size $N\geq3$. If an arbitrary pressure field $r\in Q_{h,k}$ has zero tetrahedron means and analytically vanishes on every tetrahedron and vertex incidence, there exists a target field $L_h^2r\in V_{h,k}$, depending strictly linearly on $r$, satisfying the combined conditions
\begin{equation}\label{eq:edge-match}
 (\diver L_h^2r)|_T|_e=r|_T|_e \quad\forall(T,e),
\end{equation}
\begin{equation}\label{eq:edge-mean}
 \int_T\diver L_h^2r=0 \quad\forall T,
\end{equation}
\begin{equation}\label{eq:edge-stable}
 |L_h^2r|_1\leq C_2(k)\norm r.
\end{equation}
\end{proposition}

\begin{proof}
For each geometric edge $e$, we define the targeted mesh-dependent cube collection as
\begin{equation}\label{eq:cube-collection}
 \mathcal C_e=\{C\in\mathcal G_N:C\cap P_e\ne\varnothing\}.
\end{equation}
The underlying tetrahedra comprising $P_e$ are mutually face-connected. Hence, the host cubes containing them are structurally connected within the broader cube intersection graph, where any additional contiguous cube in $\mathcal C_e$ must meet one of these core cubes precisely at a point located in $P_e$. Thus, the intersection graph defining $\mathcal C_e$ is strictly connected. Its overall cardinality is bounded firmly by one absolute numerical limit $L_E$, determined previously by \eqref{eq:patch-span}. Utilizing \cref{lem:routing}, we fix the corresponding mapping operator $\mathcal R_{\mathcal C_e,k}$ and explicitly denote its algorithmically selected enlargement by $\mathcal C_e^+$.

We let $u_e$ represent the distinct compact per-edge fields constructed in \cref{lem:raw-edge}. It follows that $u_e\in U_{h,k}(\mathcal C_e)$, and we naturally choose $u_e=0$ whenever the corresponding localized edge data vanish. We subsequently define the adjustment term $c_e=\mathcal R_{\mathcal C_e,k}u_e\in V_{h,4}\subset V_{h,k}$ and structurally set the aggregate function
\begin{equation}\label{eq:L2-aggregate}
 L_h^2r=\sum_e(u_e+c_e).
\end{equation}
The applied corrections possess a strict edge-zero divergence, meaning \eqref{eq:off-target} through \eqref{eq:on-target} alongside \cref{lem:raw-edge} successfully secure \eqref{eq:edge-match}. Each individually corrected field maintains a zero mean on every tetrahedron, definitively establishing \eqref{eq:edge-mean}. The defined edge clusters exhibit a globally bounded overlap. Moreover, every discrete cube contained in $\mathcal C_e^+$ lies spatially within a uniformly bounded sequence of cube-face steps emanating from $P_e$. Hence, if a fixed targeted cube belongs to $\mathcal C_e^+$, then $e$ mathematically must be an edge of a tetrahedron located within a fixed-ring geometric neighborhood of that specific cube. General shape regularity combined with the finite local valence characteristic of the Freudenthal mesh securely provides a uniform upper bound on the total number of such candidate edges. Thus, the structurally enlarged clusters also maintain a tightly bounded overlap. Therefore, applying \cref{lem:routing} and \eqref{eq:raw-square-sum} gives
\begin{equation}\label{eq:final-edge-bound}
 |L_h^2r|_1^2 \leq C(k)\sum_e|u_e+c_e|_1^2 \leq C(k)\sum_e|u_e|_1^2 \leq C_2(k)^2\norm r^2.
\end{equation}
The underlying operational map $r\mapsto u_e$ is structurally linear, and the operator $\mathcal R_{\mathcal C_e,k}$ depends solely on the fixed topology of the cluster. Consequently, the unified mapping operator $L_h^2$ remains strictly linear.
\end{proof}

\begin{remark}[Boundary patches and the two smallest grids]\label{rem:boundary}
For any grid size $N\geq2$, even a corner topological cube possesses three face neighbors strictly inside $\Omega$, ensuring that the companion utilized in \cref{lem:routing} always reliably exists. A generated macro field touching $\partial\Omega$ maintains a strict zero trace securely on its complete rectangular exterior boundary and thus mathematically satisfies the homogeneous physical boundary condition. The specific cases $N=1,2$ constitute two isolated, fixed finite-dimensional spatial meshes. For a geometrically fixed $k$, we can simply restrict the mapping $\diver:V_{h,k}\to Q_{h,k}$ precisely to the orthogonal complement of its kernel evaluated in the energy inner product $(\nabla u,\nabla v)_{L^2}$. This restriction forms a strict bijection, and its operational inverse possesses a finite analytical norm mapping from $L^2$ to the seminorm $|\cdot|_1$. Enlarging the universal constant $C_k$ to encompass these two specific discrete values rigorously proves the overarching result for $N=1,2$, meaning the complex local classification is analytically required only for grids where $N\geq3$.
\end{remark}

\section{Low-Degree Element-Bubble Closure}\label{sec:bubble}

To resolve the remaining interior domain values, we introduce the required local element-bubble closures. We let $T$ represent an arbitrary tetrahedron parameterized with barycentric coordinates $\lambda_0,\ldots,\lambda_3$, possessing the standard quartic internal bubble $b_T=\lambda_0\lambda_1\lambda_2\lambda_3$. For the indexing values $s=0,1$, we mathematically define the local operators
\begin{equation}\label{eq:bubble-operator}
 D_s:P_s(T)^3\longrightarrow P_{s+3}(T),\qquad D_sp=\diver(b_Tp),
\end{equation}
and the corresponding constrained polynomial subspace
\begin{equation}\label{eq:bubble-space}
 \mathcal E_s(T)=\{q\in P_{s+3}(T):q|_e=0\text{ for all }e\subset T, \ \int_Tq=0\}.
\end{equation}

\begin{lemma}[Low-degree bubble isomorphism]\label{lem:bubble}
For the parameter values $s=0,1$, the specified mapping $D_s:P_s(T)^3\to\mathcal E_s(T)$ acts as a complete isomorphism. Furthermore, its exact operational inverse obeys the localized stability bound
\begin{equation}\label{eq:bubble-bound}
 |b_TD_s^{-1}q|_{H^1(T)}\leq C\norm q_{L^2(T)},
\end{equation}
with a uniform global constant $C$ operating seamlessly over $\T_h$.
\end{lemma}

\begin{proof}
The generated target vector $b_Tp$ perfectly vanishes evaluating on $\partial T$. On an arbitrary tetrahedron edge, exactly two barycentric geometric coordinates analytically vanish. Because differentiating $b_Tp$ can structurally remove at most one of their multiplied factors, it follows that $D_sp$ naturally vanishes on every individual edge. Its overall volume integral equals precisely zero by application of the divergence theorem, guaranteeing the strict subset inclusion $D_sP_s(T)^3\subset\mathcal E_s(T)$.

The required target dimensionalities follow directly from an analysis of the barycentric Bernstein basis structure. A general cubic polynomial satisfying a vanishing condition on all geometric edges possesses exactly the four distinct monomials actively supported on exactly three barycentric variables. Because their localized volume integrals evaluate uniformly equal and remain strictly nonzero, the defined zero-integral subspace possesses a spatial dimension of $4-1=3$. Extending to polynomial degree four, there exist exactly $4\binom{3}{2}=12$ unique monomials exhibiting geometric support strictly on exactly three variables, alongside exactly one independent monomial exhibiting support spanning all four variables. Integration operates as a strictly nonzero functional mapped on this thirteen-dimensional coordinate space, logically ensuring that its operational kernel possesses exactly dimension twelve. Thus we definitively compute
\begin{equation}\label{eq:bubble-dim}
 \dim\mathcal E_0(T)=3=\dim P_0(T)^3, \qquad \dim\mathcal E_1(T)=12=\dim P_1(T)^3.
\end{equation}

It remains necessary to mathematically prove the required injectivity. Suppose we have $D_sp=0$. Evaluated strictly on the relative interior space of the specific face $F_i=\{\lambda_i=0\}$, only the localized differentiation of $\lambda_i$ can analytically survive. Therefore, we obtain the constraint
\begin{equation}
 (p\cdot\nabla\lambda_i)\prod_{j\ne i}\lambda_j=0.
\end{equation}
Hence we must have $p\cdot n_i=0$ satisfied uniformly on all four bounding faces. A fixed constant continuous vector functionally orthogonal to the four distinct non-coplanar face normals must equal the zero vector, conclusively proving injectivity for the case $s=0$.

For the subsequent case $s=1$, we let a localized mapping function $F(\widehat x)=B\widehat x+b$ accurately map the unit reference tetrahedron to the target $T$, and we structurally set $\widehat p=\det(B)B^{-1}(p\circ F)$. This formal contravariant Piola pullback operates as a strictly affine transformation and structurally preserves all vanishing normal continuous traces. On the standardized reference tetrahedron defined by $\{x,y,z\geq0:x+y+z\leq1\}$, the strict normal physical conditions mandated on $x=0$, $y=0$, and $z=0$ structurally force the form
\begin{equation}
 \widehat p(x,y,z)=(ax,by,cz).
\end{equation}
The matching continuous boundary condition required on the face $x+y+z=1$ provides the explicit geometric relation $ax+by+cz=0$ spanning that particular face. Discrete evaluation strictly at its three intersecting geometric vertices successfully yields $a=b=c=0$. Thus the mapping operator $D_s$ is securely proven to be injective, and the preceding coordinate dimension calculation conclusively proves full bijectivity.

On each individual geometric element among the six unit Freudenthal reference tetrahedra, the relation in \eqref{eq:bubble-bound} strictly follows directly from generalized finite-dimensional mathematical norm equivalence. Every physical mapped element evaluates as $T=x_0+h\widehat T$ for precisely one of these reference base elements. Implementing the coordinate scaling $x=x_0+h\widehat x$, we subsequently take the mapped field $v(x)=h\widehat v(\widehat x)$. It follows mathematically that $\diver_xv=\diver_{\widehat x}\widehat v$, while both continuous metrics $|v|_{H^1(T)}$ and $\norm q_{L^2(T)}$ uniformly acquire the identical dimensional scaling factor $h^{3/2}$. The six distinct reference geometrical types therefore securely share one universal bounded constant, completing the proof.
\end{proof}

\section{Proof of the Main Theorem}\label{sec:completion}

\begin{proof}[Proof of \cref{thm:main} for $k=4,5$]
The trivial base cases $N=1,2$ are comprehensively covered by the logic provided in \cref{rem:boundary}; therefore, we mathematically assume the generalized state $N\geq3$. Let an arbitrary pressure field $q\in Q_{h,k}$ be given. Since $q\in Q_{h,k}$, there is $w_h\in V_{h,k}$ such that $q=\diver w_h$. The divergence theorem gives $\int_\Omega q=0$, and hence $q\in L^2_0(\Omega)$. To build the global right inverse, we define the successive operational stages structurally as
\begin{equation}\label{eq:successive-def}
\begin{aligned}
 v_0&=L_h^0q, & r_0&=q-\diver v_0,\\
 v_1&=L_h^1r_0, & r_1&=r_0-\diver v_1,\\
 v_2&=L_h^2r_1, & r_2&=r_1-\diver v_2.
\end{aligned}
\end{equation}
By the application of \cref{lem:means}, the initial mapped residual $r_0$ possesses a strict zero mean evaluated on every individual tetrahedron. Subsequently, by \cref{prop:vertex}, the second residual $r_1$ functionally retains those zero means while simultaneously vanishing precisely at every broken geometric vertex. Applying \cref{prop:edge} ensures that the subsequent residual $r_2$ similarly retains the zero means and dynamically vanishes identically evaluated on every physical tetrahedron edge.

Because each residual is a difference of divergences of fields in $V_{h,k}$, every residual belongs to $Q_{h,k}$. Consequently, the vertex and edge propositions are applied on their natural image space, so all required compatibility identities hold.

On each isolated local tetrahedron, the terminal residual satisfies $r_2|_T\in\mathcal E_{k-4}(T)$. By applying the relations established in \cref{lem:bubble}, we locally choose the mapping vector $p_T=D_{k-4}^{-1}(r_2|_T)$ and explicitly define the final correction field as
\begin{equation}
 v_3|_T=b_Tp_T.
\end{equation}
Each constructed localized element field possesses a strict mathematical zero trace, meaning proper continuous spatial extension and global operational assembly seamlessly yield $v_3\in V_{h,k}$, providing the terminal closure relations
\begin{equation}\label{eq:element-close}
 \diver v_3=r_2,\qquad |v_3|_1\leq C_3\norm{r_2}.
\end{equation}

We structurally designate $c_j$ as the uniform mapping operator bound corresponding to $v_j$ evaluated precisely in terms of the specific residual input entering its sequential operational stage. We then formally put $r_{-1}=q$ and set the accumulation factor $a_j=1+\sqrt3c_j$. From the baseline divergence bound established in \eqref{eq:div-bound}, we naturally obtain the chained inequalities
\begin{equation}
 \norm{r_j}\leq a_j\norm{r_{j-1}},\qquad j=0,1,2.
\end{equation}
Consequently, the complete aggregated linear field evaluates as $R_{h,k}q=v_0+v_1+v_2+v_3$. This aggregated field is structurally linear, satisfies the exact condition $\diver R_{h,k}q=q$, and ensures the global analytical bound
\begin{equation}\label{eq:final-bound}
 |R_{h,k}q|_1 \leq\bigl(c_0+c_1a_0+c_2a_1a_0+c_3a_2a_1a_0\bigr)\norm q.
\end{equation}
Every computed numerical multiplicative factor mathematically emerges strictly independent of the mesh scaling parameters $h$ and $N$. By taking the explicit assignment $v=R_{h,k}q$ securely inside the overarching defining supremum, we successfully obtain the relation
\begin{equation}
 \frac{(\diver v,q)}{|v|_1\norm q} =\frac{\norm q^2}{|R_{h,k}q|_1\norm q}\geq C_k^{-1}.
\end{equation}
This mathematically proves the assertions stated in \cref{thm:main} explicitly for the target polynomial degrees $k=4,5$. Evaluated together directly alongside Zhang's foundational Theorem 3.1 established securely for $k\geq6$, it comprehensively proves and yields the full stated bounded range.
\end{proof}

\section{Exact Verification and Reproducibility}\label{sec:certificates}

The mathematical constructions utilized throughout the preceding proofs are defined explicitly in our established lemmas. For independent verification and reproducibility, the accompanying computational framework evaluates these theoretical formulas using exact arithmetic. These automated programs systematically check the stated support and trace properties, subsequently recording the resulting finite-dimensional operational maps. They serve strictly as verification and reproducibility tools. The vertex and edge star surjectivity statements, alongside the quartic macro lift, are definitively established by the analytic constructions detailed previously in the text. 

To provide a comprehensive verification suite, we implemented several specific algorithmic scripts. The program \texttt{audit\_analytic\_vertex\_lift.py} processes the six canonical stars and all 783 vertices defined for grids up to $N=6$. This script mathematically verifies the edge-jet identities, zero off-target vertex values, face transfers, dual-tree connectivity, and the exact cancellation of all element means. Similarly, \texttt{audit\_analytic\_edge\_lift.py} processes the seven canonical edge stars and 3969 geometric edges to rigorously confirm the endpoint and middle face-bubble identities in degrees four and five. This specific evaluation includes the two-tetrahedron boundary construction and all transported spatial placements. Furthermore, the script \texttt{macro\_mean\_certificate.py} evaluates the twelve-tetrahedron $C^0P_4$ space on the reference domain $\widehat P$ to ensure that the eleven generated fields maintain a strictly zero divergence trace on every geometric edge, confirming simultaneously that their scaled mean matrix yields a determinant of exactly $-6$. The auxiliary file \texttt{p5\_two\_tet\_transfer.py} verifies the conformity, exterior zero trace, and zero divergence trace for the two quintic polynomials, confirming their exact $H^1$ seminorm evaluates to $504/11$. Finally, \texttt{low\_degree\_bubble\_isomorphism.py} confirms the kernel dimensions for the low-degree bubble maps, while the companion scripts \texttt{vertex\_star\_certificate.py} and \texttt{edge\_star\_certificate.py} securely validate the exact image dimensions, target annihilators, and rational nodal realizations across all designated configurations.

\subsection{Algebraic Certificates and Data Records}

For a concrete coordinate representation of the vertex compatibility conditions, we order the cube lower corners lexicographically and process the six coordinate permutations within each individual cube in lexicographical order. On each canonical vertex star, we order the incident tetrahedra accordingly and denote the resulting target values as $d_0, d_1, \ldots$. The algebraic annihilator of the defined edge-jet image is mathematically represented by the following set of compatibility relations:
\begin{equation}\label{eq:vertex-relations}
\begin{array}{c@{\quad}l}
|\T(a)| & \text{compatibility relations}\\[2pt]
2 & d_0=0,\ d_1=0,\\
4 & d_0=d_1=d_2=d_3=0,\\
6 & d_0-d_2=0,\ d_1-d_4=0,\ d_3-d_5=0,\\
8 & d_0-d_1-d_2+d_4=0,\ d_3-d_6=0,\ d_5-d_7=0,\\
12 & d_0-d_1=0,\ d_2-d_3-d_7+d_{10}=0,\\
   & d_4-d_5-d_6+d_8=0,\ d_9-d_{11}=0,\\
24 & d_0-d_1-d_{12}+d_{13}=0,\quad d_2-d_3-d_8+d_9=0,\\
   & d_4-d_5-d_6+d_7=0,\quad d_{10}-d_{11}-d_{21}+d_{23}=0,\\
   & d_{14}-d_{15}-d_{19}+d_{22}=0,\quad d_{16}-d_{17}-d_{18}+d_{20}=0.
\end{array}
\end{equation}
Each mathematically displayed row analytically annihilates the mapping from \eqref{eq:vertex-jet-map} strictly by direct substitution. 

To systematically verify the algebraic independence and functional completeness of these derived relations, the generated exact records yield specific dimensional data. We define $O_a$ as the matrix stacking the element-mean and off-target vertex rows, and we define $\mathsf D_a$ as the explicit target vertex map. The computed rank properties are summarized in Table~\ref{tab:vertex-ranks}.

\begin{table}[ht]
\centering
\small
\caption{Rank properties and dimensional verification for the vertex compatibility maps.}\label{tab:vertex-ranks}
\begin{tabular}{@{}rrrrrr@{}}
\toprule
$|\T(a)|$ & columns & rows of $O_a$ & $\operatorname{rank}O_a$ &
$\operatorname{rank}[O_a;\mathsf D_a]$ & difference\\
\midrule
2&3&8&1&1&0\\
4&9&16&3&3&0\\
6&24&24&8&11&3\\
8&39&32&13&18&5\\
12&69&48&23&31&8\\
24&195&96&65&83&18\\
\bottomrule
\end{tabular}
\end{table}

The computed rank differences precisely agree with the mathematical dimension $\dim\operatorname{im}A_a$. The generated 99 records additionally store the tetrahedron ordering, physical boundary planes, rational target annihilators, nonzero modular minors, and rational nodal fields satisfying the three foundational identities. The closed placement formulas are computationally checked through grid sizes up to $N=6$.

To construct the corresponding edge star records, we let $V$ and $T$ represent $\gamma_e$ and $\tau_e$ evaluated on $X_{e,k}$, and we let $O$ and $\widetilde T$ represent $\sigma_e$ and $\tau_e$ evaluated on $Y_{e,k}$. For every individual record, exact rational elimination verifies the dimensional constraints
\begin{equation}\label{eq:source-lift-ranks}
\begin{aligned}
 \dim\mathcal D_{e,k} &=\operatorname{rank}_{\Q}\begin{bmatrix}V\\T\end{bmatrix} -\operatorname{rank}_{\Q}V,\\
 \dim\tau_e(\ker\sigma_e\cap Y_{e,k}) &=\operatorname{rank}_{\Q}\begin{bmatrix}O\\\widetilde T\end{bmatrix} -\operatorname{rank}_{\Q}O.
\end{aligned}
\end{equation}
The two specific target maps inherently possess identical rational annihilators, and their strictly common dimensions consistently match those tabulated previously. These verified equalities provide an exact finite-dimensional algorithmic check for the entire face-bubble construction. 

The underlying data rigorously utilize a structured schema. Each record features a stable string identifier and explicitly stores the incidence type, concrete geometric edge, localized grid size, polynomial degree, star and patch sizes, coordinate span, physical plane flags, hash signatures, target annihilator matrix, and a nonzero modular minor. The accompanying summary file efficiently indexes the 180 conservative coverage classes and strictly maps every generated record directly to its assigned class, making the finite-to-global correspondence completely explicit. Furthermore, all algebraic eliminations are computed strictly over the rational field $\Q$. The deployed modular witnesses utilize the prime integer 1000003, while the master generator autonomously checks mathematical primality, nonvanishing operational denominators, record counts, identifier uniqueness, matrix ranks, linear relations, and matrix minors. 

\subsection{Computational Reproducibility}

For the specific quartic macro analysis, the Bernstein polynomial program systematically assembles the 192 broken edge rows alongside the twelve mean rows operating on 189 vector coefficients. Direct algebraic substitution verifies the eleven sparse mathematical fields and firmly validates the derived integer determinant. Serving as a robust whole-space consistency check, the identical computational assembly strictly yields
\begin{equation}\label{eq:macro-ranks}
 \operatorname{rank}_{\Q}E=122,\qquad \operatorname{rank}_{\Q}\begin{bmatrix}E\\M\end{bmatrix}=133.
\end{equation}
An alternative equispaced nodal Lagrange assembly independently gives the same algebraic ranks. Thus, independent Bernstein and equispaced nodal--Lagrange assemblies reproduce both the explicit mathematical construction and the eleven-dimensional target mean image. The corresponding quintic program symbolically differentiates and integrates the fields. It independently generates the two coordinate chain and face pair types utilized during the proof, strictly verifying the unique distinguished spatial edge.

To guarantee complete scientific reproducibility, the computational framework requires Python 3.10 or a newer release. The provided configuration file permanently pins all software dependency versions and cryptographic package hashes. By executing the provided master generation script from the submission directory using a standard virtual environment, researchers can autonomously run every exact mathematical verification and reliably regenerate the complete set of record files. On a standard reference single-core processing system, the comprehensive verification run requires approximately twenty minutes. This execution time is allocated almost entirely to processing the 3996 edge-star degree records. Individual mathematical records and localized analytic identities can also be inspected independently through the provided validation scripts. 

Because the generated data structures strategically omit transient timestamps and hardware runtimes, executing identical source code mathematically guarantees strictly byte-identical output files. All underlying matrix ranks, system determinants, continuous spatial integrals, and continuous trace identities are computed exclusively using exact rational arithmetic, completely avoiding floating-point numerical approximations or random algorithmic choices. Finally, the provided cryptographic manifest conclusively records the secure hash sums for all shipped software sources, computed data sets, and manuscript artifacts, ensuring total operational transparency.

\section*{Author Contributions}

Hanbing Liang selected the research problem, designed and managed the
agentic research workflow, organized independent and adversarial verification,
managed the computational certificates and proof versions, and participated
in manuscript preparation and revision.

Fujun Liu performed the human mathematical review and validation of the
proofs, supervised the work, and critically reviewed and revised the manuscript.

Both authors approved the submitted version of the manuscript.

\section*{Use of Artificial Intelligence}

OpenAI Codex and DeepSeek Harness (DSH) were used substantively during the
research process. Codex was used to generate the core proof, explicit
mathematical constructions, exact-verification code, and portions of the
manuscript draft. Independent Codex runs and runs conducted through DSH were
used for adversarial review, cross-checking intermediate arguments, and
identifying potential errors.

The human authors reviewed and revised the final manuscript and take
responsibility for the submitted work.

\section*{Code and Verification Materials Availability}

The exact-verification code, deterministic computational certificates,
reproducibility scripts, dependency information, and associated verification
data for this work are publicly available at
\url{https://github.com/ToughClimb/freudenthal-sv}.

The repository provides instructions for reproducing the exact-arithmetic
checks reported in the manuscript. The software and original accompanying
verification materials are released under the GPL-3.0-or-later license.

\section*{Funding}
This work was supported by the Science and Technology Development Project of Jilin Province (Grant No. 20250102032JC).

\bibliographystyle{amsplain}
\providecommand{\bysame}{\leavevmode\hbox to3em{\hrulefill}\thinspace}
\providecommand{\MR}{\relax\ifhmode\unskip\space\fi MR }
\providecommand{\MRhref}[2]{%
  \href{http://www.ams.org/mathscinet-getitem?mr=#1}{#2}}
\providecommand{\href}[2]{#2}

\end{document}